\documentclass{scrartcl}
\usepackage[utf8]{inputenc}
\usepackage{hyperref}
\usepackage{url}
\usepackage[english]{babel}
\usepackage{enumerate}
\usepackage{aliascnt}
\usepackage{amssymb}
\usepackage{amsmath}
\usepackage{amsthm}
\usepackage{verbatim}
\usepackage{mathtools}
\usepackage{microtype}

\usepackage[nottoc]{tocbibind} 
\usepackage[title]{appendix}

\usepackage{color}

\numberwithin{equation}{section} 

\usepackage{dsfont} 
\usepackage{hyperref} 
\usepackage[noabbrev]{cleveref}

\usepackage{autonum} 

\allowdisplaybreaks[4] 

\newaliascnt{lem}{prop} 

\newtheorem{lem}[lem]{Lemma}

\aliascntresetthe{lem} 
\Crefname{lem}{Lemma}{Lemmas}

\newaliascnt{defi}{prop} 
 \newtheorem{defi}[defi]{Definition}

\aliascntresetthe{defi} 
\Crefname{defi}{Definition}{Definitions}

\newaliascnt{cor}{prop} 
 
\aliascntresetthe{cor}

\newaliascnt{remark}{prop} 
 \newtheorem{remark}[remark]{Remark}
 
\aliascntresetthe{remark} 

\newaliascnt{thm}{prop} 
 \newtheorem{thm}[thm]{Theorem}
 
\aliascntresetthe{thm} 

\newaliascnt{example}{prop} 
 
\aliascntresetthe{example}

\def\equationautorefname~#1\null{%
  (#1)\null
}

\DeclareMathOperator{\supp}{supp} 
 
\DeclareMathOperator{\dist}{dist} 
\DeclareMathOperator*{\esssup}{ess\,sup}
\newcommand{\R}{\ensuremath{\mathbb{R}}}
\newcommand{\N}{\ensuremath{\mathbb{N}}}

\renewcommand{\S}{\ensuremath{\mathbb{S}}}

\newcommand{\CalB}{\ensuremath{\mathcal{B}}}

\newcommand{\CalD}{\ensuremath{\mathcal{D}}}

\newcommand{\CalL}{\ensuremath{\mathcal{L}}}
\newcommand{\CalO}{\ensuremath{\mathcal{O}}}

\newcommand{\CalM}{\ensuremath{\mathcal{M}}}
\newcommand{\CalX}{\ensuremath{\mathcal{X}}}
\newcommand{\CalY}{\ensuremath{\mathcal{Y}}}
\newcommand{\CalZ}{\ensuremath{\mathcal{Z}}}
\newcommand{\CalU}{\ensuremath{\mathcal{U}}}
\newcommand{\CalV}{\ensuremath{\mathcal{V}}}
\newcommand{\CalK}{\ensuremath{\mathcal{K}}}
\newcommand{\CalJ}{\ensuremath{\mathcal{J}}}

\DeclareMathOperator{\Id}{Id}

\DeclareMathOperator{\diverg}{div}
\DeclareMathOperator{\curl}{curl}

\title{Statistical solutions of the incompressible Euler equations}
\author{Raphael Wagner\thanks{Institute of Applied Analysis, Ulm University, Helmholtzstra\ss e 18, 89081 Ulm, Germany. \texttt{raphael.wagner@uni-ulm.de}}\and
Emil Wiedemann\thanks{Institute of Applied Analysis, Ulm University, Helmholtzstra\ss e 18, 89081 Ulm, Germany. \texttt{emil.wiedemann@uni-ulm.de}}}

\begin{document}

\maketitle
\begin{abstract}
\noindent\textbf{Abstract:} We study statistical solutions of the incompressible Euler equations in two dimensions with vorticity in $L^p$, $1\leq p \leq \infty$, and in the class of vortex-sheets with a distinguished sign. Our notion of statistical solution is based on the framework due to Bronzi, Mondaini and Rosa in \cite{Bronzi2016}. Existence in this setting is shown by approximation with discrete measures, concentrated on deterministic solutions of the Euler equations. Additionally, we provide arguments to show that the statistical solutions of the Euler equations may be obtained in the inviscid limit of statistical solutions of the incompressible Navier-Stokes equations. Uniqueness of trajectory statistical solutions is shown in the Yudovich class. 
\end{abstract}

\section{Introduction and main results}
We consider the \emph{incompressible Euler equations} in two dimensions
\begin{equation}\label{eq:Euler vel}
	\begin{array}{r}
		\partial_t u + \diverg(u\otimes u) + \nabla p = 0\\
		\diverg u = 0
	\end{array}
\end{equation}
for a time-parametrized velocity field $u = (u^1(x,t),u^2(x,t))^t: \R^2 \times (0,T) \to \R^2$ and the scalar pressure $p:\R^2 \times (0,T) \to \R$, where $T > 0$ is fixed.\\
In the classic literature, assumptions on the initial vorticity $\omega(u) := \curl u = \partial_1 u^2 - \partial_2 u^1$ yielded some of the first existence and uniqueness results for the corresponding Cauchy problem. The result of existence and uniqueness for weak solutions with vorticity in $L^\infty$ is due to Yudovich (see \cite{Yudovich1963}). Existence of weak solutions with vorticity in $L^p$, $1 < p < \infty$, was first proved by DiPerna and Majda in \cite{DiPerna1987}. Weak solutions in the class of vortex-sheets with a distinguished sign, i.e.\ bounded and positive Radon measures in the negative Sobolev space $H^{-1}$ were constructed by Delort in \cite{Delort1991}. Delort's arguments were later used by Vecchi and Wu in \cite{Vecchi1993} to also show existence of weak solutions with vorticity in the class $L^1 \cap H^{-1}$. While there are many more classes that one could consider, these are the ones that we will restrict ourselves to in this article.\\
The problem of uniqueness has still not been resolved in the aforementioned cases, with exception of the case of uniformly bounded vorticity, considered by Yudovich. Let us also point here towards the recent preprints by Vishik (\cite{Vishik2018_01}, \cite{Vishik2018_02}), where non-uniqueness of the forced Euler equations with vorticity in some $L^p$ space has been shown, or the work of Albritton, Bru\'{e} and Colombo in \cite{Albritton2021}, in which non-uniqueness of Leray-Hopf solutions of the forced Navier-Stokes equations has been proved. A lack of well-posedness and in some way uncertainty of the evolution of an initial state may lead one to the approach of considering a whole ensemble of weak solutions which may display desired properties in a statistical sense. This concept of considering ensembles and ensemble averages is common in the theory of turbulence, where individual flows may exhibit wild behaviour, while quantities, averaged in space and time, appear to be sort of universal over an ensemble of flows (see e.g. \cite{Frisch1995}).\\
A mathematically rigorous concept of statistical solutions in the context of the two and three-dimensional Navier-Stokes equations has been introduced in the seventies by Foias in \cite{Foias1972} and \cite{Foias1973} based on discussions with Prodi. Another concept of statistical solutions of the Navier-Stokes equations is due to Fursikov and Vishik (see \cite{Vishik1988}). The two notions differ in the way that Foias considered time parametrized probability measures on the phase space, while Vishik and Fursikov constructed single probability measures on an appropriate space of trajectories.\\
A connection between the two approaches was drawn not too long ago by Foias, Rosa and Temam in \cite{Foias2013}, where they consider probability measures on the space of weak (Leray-Hopf) solutions of the Navier-Stokes equations in the spirit of Vishik and Fursikov and show that the projections in time of these measures are statistical solutions, as originally considered by Foias. The former measures were then called Vishik-Fursikov measures by the authors and the resulting statistical solutions, obtained by projecting these in time, Vishik-Fursikov statistical solutions.
The proof of existence of Vishik-Fursikov statistical solutions in \cite{Foias2013} also differed significantly from the one given in \cite{Foias1972}, in which approximating sequences of probability measures were used, based on Galerkin approximations of the deterministic equations. The new proof, originally presented in \cite{Foias2001}, uses approximations by discrete measures, constructed using the Krein-Milman theorem. This latter proof has the flexibility of being generalized and applied to other equations. The theory of statistical solutions of the Navier-Stokes equations has in fact served as motivation for Bronzi, Mondaini and Rosa to develop an abstract concept of statistical solution in \cite{Bronzi2016}, in which, to some extent, not even an underlying partial differential equation is required. In their work, they label the analogue of the Vishik-Fursikov measures \textit{trajectory statistical solutions} and the analogue of the aforementioned statistical solutions is labelled more specifically \textit{phase space statistical solution}.\\ Though we will not require the same generality kept in \cite{Bronzi2016}, we will use their framework as a unifying concept for each class of weak solutions of the Euler equations mentioned earlier in this introduction.\\
The main results we prove in this work are essentially the following three: For each of the aforementioned classes of weak solutions of the Euler equations in the phase space $X$ and every $T > 0$, given a distribution $\mu_0$ on an appropriate space of initial data $X_0 \subset X$,
\begin{enumerate}[I)]
\item there exists a trajectory statistical solution $\rho$ satisfying $\Pi_0\rho = \mu_0$, where $\Pi_t$ is the time evaluation map at time $t$ for every $0 \leq t \leq T$.\\ 
We also have energy type inequalities for the velocity and vorticity:
\[
\int_{C([0,T];X)}\|u_{kin}(t)\|_{L^2(\R^2)}\,d\rho(u) \leq \int_X a^{|m(u_0)|}\|u_{0,kin}\|_{L^2(\R^2)}\,d\mu_0(u_0)
\]
for almost every $0 \leq t \leq T$, a constant $a > 1$ and
\begin{equation}\label{eq: dissip Lp vorticity}
\int_{C([0,T];X)} \|\omega(u(t))\|_{X_{vort}}\,d\rho(u) \leq \int_{X} \|\omega(u_0)\|_{X_{vort}}\,d\mu_0(u_0)
\end{equation}
for every $0 \leq t \leq T$. In case that $X_{vort} = L^p$, $1< p \leq \infty$, equality holds in \eqref{eq: dissip Lp vorticity}. In the Yudovich class, this trajectory statistical solution is unique. 
\item If $\int_{X_0} \gamma(u_0)^2\,d\mu_0(u_0) < \infty$, where $\gamma(u_0) := a^{|m(u_0)|}(1+\|u_{0,kin}\|_{L^2(\R^2)})$, then the projections $\lbrace \rho_t := \Pi_t\rho \rbrace_{0 \leq t \leq T}$ form a phase space statistical solution with $\rho_0 = \mu_0$, so that
\begin{equation}\label{eq: Foias-Liouville}
\int_X \Phi(u)\,d\rho_t(u) = \int_X \Phi(u)\,d\rho_{t'}(u) + \int_{t'}^t\int_X \langle (u(s)\otimes u(s)),\nabla\Phi'(u)\rangle_{L^2(\R^2)}\,d\rho_s(u)\,ds
\end{equation}
for every $0 \leq t' \leq t \leq T$ and every appropriate test functional $\Phi$. 
Moreover
\[
\int_{X}\|u_{kin}\|_{L^2(\R^2)}\,d\rho_t(u) \leq \int_{X}a^{|m(u_0)|}\|u_{0,kin}\|_{L^2(\R^2)}\,d\mu_0(u_0)
\]
for almost every $0 \leq t \leq T$ and
\begin{equation}\label{eq: dissip Lp vorticity 2}
\int_{X} \|\omega(u)\|_{X_{vort}}\,d\rho_t(u) \leq \int_{X} \|\omega(u_0)\|_{X_{vort}}\,d\mu_0(u_0)
\end{equation}
for every $0 \leq t \leq T$. In case that $X_{vort} = L^p$, $1< p \leq \infty$, equality holds in \eqref{eq: dissip Lp vorticity 2} 
\item the statistical solutions may be constructed in a weak sense as the inviscid limit of a sequence of statistical solutions of the Navier-Stokes equations, which attain the initial distribution $\mu_0$ in the same sense. 
\end{enumerate}
For I), we remark that the case of weak solutions of the Navier-Stokes equations and of the Euler equations with uniformly bounded vorticity is special due to uniqueness of the weak solutions and we will treat their construction separately.\\ 
The idea of the abstract proof of existence of trajectory statistical solutions in \cite{Bronzi2016} will then, purely for didactic purposes, be displayed in the particularly simple situation of the vorticity being in $L^p(\R^2),$ $1 < p < \infty$, where one may utilize strong compactness properties. We could, however, as we will do in all other cases, directly refer to the existence results in \cite{Bronzi2016} and only check the required conditions.\\
Likewise, for existence of phase space statistical solutions as just stated in II), we will demonstrate the abstract arguments in \cite{Bronzi2016} for the specific case of vorticity being in $L^\infty(\R^2)$ and note that all other cases may be treated similarly.\\
Since trajectory statistical solutions of the Navier-Stokes equations are essentially probability measures having support in the set of weak solutions of the Navier-Stokes equations, the discussion of the inviscid limit of those trajectory statistical solutions will be reduced to the study of the behaviour of the support and in essence, the deterministic inviscid limit. The inviscid limit results for phase space statistical solutions that we are going to state are then going to be the direct consequences of the results for trajectory statistical solutions.\\
We mention here that statistical solutions of the Euler equations have been considered in similar ways before by D. Chae in \cite{Chae1991p} and \cite{Chae1991g}, P. Constantin and J. Wu in \cite{Constantin1997} and J. Kelliher in \cite{Kelliher2009}. Chae constructs (phase space) statistical solutions of the Euler equations on the two-dimensional torus in \cite{Chae1991p} under an assumption on the mean enstrophy with respect to the initial distribution, which corresponds to the situation of solutions of the Euler equations with vorticity in $L^p$ for $p = 2$. Our work here can be thought of as a generalization as we allow for any $1 \leq p \leq \infty$. A lesser difference lies in our article using $\R^2$ as the underlying domain. More fundamentally, our work differs in the used constructions. Chae uses the rather sophisticated compactness arguments that were used in Foias' original article in \cite{Foias1972} to obtain statistical solutions of the Euler equations by means of a vanishing viscosity argument. These arguments include the Banach-Alaoglu theorem, applied to a family of functionals which are related to the statistical solutions of the Navier-Stokes equations, in combination with results from the theory of the Daniell integral and the theory of lifting to recover a suitable family of measures.\\
While we have also included an inviscid limit argument, our focus lies on the results and ideas from \cite{Bronzi2016}, in which discrete approximations are used, given by the Krein-Milman theorem.\\
Constantin and Wu consider statistical solutions of the Navier-Stokes and of the Euler equations with uniformly bounded vorticity on $\R^2$, i.e.\ the case $p = \infty$. In contrast to this and the other articles previously mentioned, they work with the vorticity equation opposed to the velocity formulation. Also, they consider both phase space as well as trajectory statistical solutions for both the Euler and Navier-Stokes equations. Their proof of existence for trajectory statistical solutions of the Euler equations is in close proximity to the work of Vishik and Fursikov who also relied on Prokhorov's theorem to construct trajectory statistical solutions by an inviscid limit argument. For phase space statistical solutions, they used similar arguments as Chae to obtain these by an inviscid limit argument.\\
Kelliher also considers phase space statistical solutions of the Navier-Stokes and of the Euler equations in the Yudovich class but in their velocity formulation. The phase space statistical solution of the Euler equations is also obtained by an inviscid limit argument. However, while Constantin and Wu immediately employ the solution operator for the two-dimensional Navier-Stokes equations to construct corresponding statistical solutions, Kelliher first considers phase space statistical solutions of the Navier-Stokes equations on balls of finite radius as constructed by Foias and then constructs statistical solutions of the Navier-Stokes equations on $\R^2$ by employing the theory of the expanding domain limit for the deterministic Navier-Stokes equations.\\
Unlike Chae or Constantin and Wu, who make assumptions on the $L^2$ or $H^1$ norm of the mean vorticity with respect to the initial distribution, we will, similarly to Kelliher, consider infinite energy solutions and require a bound on the mean total mass of vorticity and the $L^2$ norm of the finite kinetic energy parts with respect to the initial distribution. This is necessary for the Foias-Liouville equation \eqref{eq: Foias-Liouville} of phase space statistical solutions to be well-defined due to the connection between the (local) $L^2$ norm and the total mass of vorticity.\\
We will begin by giving an overview of the well-known results of existence and uniqueness of weak solutions in each of the aforementioned classes and highlight certain properties.\\
In the second part of the introduction, we briefly review the abstract framework of statistical solutions given in \cite{Bronzi2016}.\\
In the main part, we construct statistical solutions in each class using the properties of weak solutions and by applying the ideas and results for abstract statistical solutions, both of which are described in the two preliminary sections.

\subsection{Preliminaries on the Euler and Navier-Stokes equations}\label{sec: euler}
We denote by $L^p(\R^2)$ and $L^p(\R^2;\R^2)$, $1 \leq p \leq \infty$, the standard Lebesgue spaces on $\R^2$ with values in $\R$ and $\R^2$ respectively. The associated standard $L^p$ Sobolev spaces of order $s$ on $\R^2$ will be denoted by $W^{s,p}(\R^2)$ and $W^{s,p}(\R^2;\R^2)$ for all $s \in \N, 1 \leq p \leq \infty$. For the specific case of $p = 2$, we let $H^s(\R^2) := W^{s,2}(\R^2)$ and $H^s(\R^2;\R^2) := W^{s,2}(\R^2;\R^2)$ for all  $s \in \N$. The dual of $H^s(\R^2)$ and $H^s(\R^2;\R^2)$ will be denoted by $H^{-s}(\R^2)$ and $H^{-s}(\R^2;\R^2)$ respectively.\\
Related spaces of local integrability or compact support will generally be denoted by the addition of the subscript \emph{loc} or \emph{c} respectively.\\
We say that a vector field $v \in L^1_{loc}(\R^2;\R^2)$ is weakly divergence-free if for every $\varphi$ in the class of smooth and compactly supported functions $C_c^\infty(\R^2)$, we have $\int_{\R^2} v\cdot \nabla\varphi\,dx = 0$. The subset of weakly divergence-free vector fields of $L^2(\R^2;\R^2)$ will be denoted by $H$.\\
In our notation for the standard norms on these spaces, we will not differ between spaces of functions with values in $\R$ or $\R^2$, e.g. $\|\cdot\|_{L^2(\R^2)}$ will denote the norm on both $L^2(\R^2)$ and $L^2(\R^2;\R^2)$.\\
Let us also point out here that $T > 0$ will be a fixed final time for all equations considered throughout this article.

\begin{defi}\label{def: Euler weak vel}
A weakly divergence-free vector field $u \in L^\infty(0,T;H_{loc})$ is called a weak solution of the Euler equations with initial data $u_0 \in H_{loc}$ if for any $v = (v^1,v^2)^t \in C_c^\infty(\R^2;\R^2) \cap H$, we have
\begin{enumerate}[i)]
\item \begin{equation}\label{eq: Euler weak vel equ}
\frac{d}{dt}\int_{\R^2}v\cdot u\,dx = \int_{\R^2} \nabla v : (u \otimes u)\,dx
\end{equation}
in the sense of distributions on $(0,T)$, where $(u \otimes u) = (u_iu_j)_{i,j=1,2}$, $\nabla v = \left(\frac{\partial v^i}{\partial_j} \right)_{i,j = 1,2} \in \R^{2\times 2}$ and $:$ denotes the Frobenius inner product between two matrices;
\item $u \in C([0,T];H^{-L}_{loc}(\R^2;\R^2))$ for some $L > 1$ and $u(0) = u_0$ in $H^{-L}_{loc}(\R^2;\R^2)$.
\end{enumerate}
If $u$ only satisfies i), then we simply say that $u$ is a weak solution of the Euler equations. 
\end{defi}

\begin{defi}\label{def: NSE weak vel}
A weakly divergence-free vector field $u \in L^\infty(0,T;H_{loc})$, for which $\nabla u \in L^2(0,T;L^2(\R^{2\times 2}))$, is called a weak solution of the Navier-Stokes equations with viscosity $\nu > 0$ and initial data $u_0 \in H_{loc}$ if for any $v = (v^1,v^2)^t \in C_c^\infty(\R^2;\R^2) \cap H$, we have
\begin{enumerate}[i)]
\item \begin{equation}\label{eq: NSE weak vel equ}
\frac{d}{dt}\int_{\R^2}v\cdot u\,dx = \int_{\R^2} \nabla v : (u \otimes u)\,dx + \nu \int_{\R^2} \Delta v \cdot u\,dx
\end{equation}
in the sense of distributions on $(0,T)$;
\item $u \in C([0,T];H_{loc})$ and $u(0) = u_0$ in $H_{loc}$.
\end{enumerate}
If $u$ only satisfies i), then we simply say that $u$ is a weak solution of the Navier-Stokes equations (with viscosity $\nu$). 
\end{defi}

As indicated here, throughout this article we will use $\R^2$ as the underlying domain. This is oftentimes convenient for two reasons. First, one may ignore boundary conditions and effects. Second, the Biot-Savart law, which constitutes how to recover velocity from vorticity, has an explicit form. In general, for a vector field $u = (u^1,u^2)^t \in L^1_{loc}(\R^2;\R^2)$, we define the (distributional) vorticity $\omega(u) = \partial_1 u^2 - \partial_2 u^1$.\\
The downside is that for a velocity field with compactly supported vorticity and given by the Biot-Savart law to have finite kinetic energy, its total mass of vorticity necessarily needs to vanish. We take care of this by only considering those vector fields which have a \emph{radial-energy decomposition}. Now following \cite{Chemin1998}[Chapter 1], for any $m \in \R$, let $E_m$ be the set of weakly divergence-free vector fields $u \in H_{loc}$
such that there exists a \emph{stationary vector field} $\sigma \in C^\infty(\R^2;\R^2)$, whose total mass of vorticity $\int_{\R^2} \omega(\sigma)$ is equal to $m$, for which
\begin{equation}\label{eq: rad-energy decomp}
u - \sigma \in H.
\end{equation}
Here, a vector field $\sigma \in C^\infty(\R^2;\R^2)$ is called stationary if there exists $g \in C^\infty_c(\R)$ such that
\[\sigma(x) = \frac{x^\perp}{|x|^2}\int_0^{|x|} s g(s)\,ds = \frac{1}{|x|^2}\begin{pmatrix}
-x^2\\
x^1
\end{pmatrix}
\int_0^{|x|} sg(s)\,ds, x = \begin{pmatrix}
x^1\\
x^2
\end{pmatrix} \in \R^2.\] 
The stationary vector fields are smooth exact solutions of the Euler equations and in the situation above, $\omega(\sigma)(x) = g(|x|), x \in \R^2$. A stationary vector field $\sigma$ also satisfies the decay properties $\sigma \in \CalO\left(\frac{1}{|x|}\right)$ and $\nabla\sigma \in \CalO\left(\frac{1}{|x|^2}\right)$. While this implies $\nabla \sigma \in L^2(\R^{2\times 2})$, it is not enough for $\sigma$ to be square integrable. However, $\sigma$ is at least bounded and vanishes at infinity.\\
Consequently, $u = \sigma + (u - \sigma)$ as above is the sum of a smooth exact solution of the Euler equations with radially symmetric vorticity and a vector field of finite kinetic energy.\\
As indicated before, $\sigma \in H$ if and only if $\int_{\R^2} \omega(\sigma) = 0$ so that $E_m = \sigma + H$ is an affine space.\\
Throughout this article, we fix one stationary vector field $\Sigma$ satisfying 
\begin{equation}\label{eq: fixed stat vec field}
\int_{\R^2} \omega(\Sigma)\,dx = 1
\end{equation}
so that $E_m = (m\Sigma) + H$. By fixing $\Sigma$, the decomposition
\begin{equation}\label{eq: unique rad-energy decomp}
u = m\Sigma + (u - m\Sigma)
\end{equation}
is now unique for any $u \in \mathbb{E}$ and we will occasionally write $u_{kin} := u - m\Sigma$ for the finite kinetic energy part of this decomposition and $m(u) := \int_{\R^2}\omega(u-u_{kin})$ for the total mass of vorticity of the stationary field. Moreover, if $\omega(u)$ is a finite Borel measure in $H^{-1}(\R^2)$ so that $(1+|x|)\omega(u)$ is also a finite Borel measure, then as a consequence of the proof of \cite{Chemin1998}[Lemma 1.3.1], we may even tell $m(u) = \int_{\R^2} \omega(u)$.\\
We then also introduce the space
\[\mathbb{E} := \bigcup_{m\in\R}E_m.\]
We note that $\mathbb{E}$ is in fact a separable Banach space with norm $\|\cdot\|_{\mathbb{E}}$ given by $\|u\|_{\mathbb{E}} = |m(u)| + \|u_{kin}\|_{L^2(\R^2)}$ for all $u \in \mathbb{E}$. By having fixed $\Sigma$, the radial energy decomposition is unique and $\|\cdot\|_{\mathbb{E}}$ well-defined.\\
\hfill\\
Let us also briefly recall that for a weak solution $u \in L^\infty(0,T;H_{loc})$ of the Euler equations and $v \in C_c^\infty(\R^2;\R^2) \cap H$, the function $t \mapsto \int_{\R^2} v\cdot u(t)\,dx$ is absolutely continuous on $[0,T]$ with 
\[\frac{d}{dt}\int_{\R^2} v\cdot u(t)\,dx = \int_{\R^2} \nabla v : (u(t)\otimes u(t))\,dx\]
for almost every $0 \leq t \leq T$.\\
\hfill\\
The following theorem contains the classic results of all four classes of functions in which we consider weak solutions of the Euler equations. By $\mathcal{M}^+(\R^2)$ and $\CalM(\R^2)$, we mean the spaces of finite (non-negative) Borel measures on $\R^2$. Oftentimes, we will also identify the latter space with the dual of the separable Banach space $C_0(\R^2)$ of continuous functions, vanishing at infinity.\\
In the following, for two Banach spaces $(X_1,\|\cdot\|_{X_1})$ and $(X_2,\|\cdot\|_{X_2})$, we define the norm $\|\cdot\|_{X_1 \cap X_2} := \|\cdot\|_{X_1} + \|\cdot\|_{X_2}$ on $X_1 \cap X_2$.

\begin{thm}\label{thm: solutions of euler}
Let $u_0 \in E_m$ for some $m \in \R$ and fix $1 < p < \infty$.
\begin{enumerate}[A)]
\item Yudovich \cite{Yudovich1963}: If $\omega(u_0) \in L^1(\R^2) \cap L^\infty(\R^2)$, then there exists a unique weak solution $u \in C([0,T];E_m)$ with vorticity $\omega(u) \in L^\infty(0,T;L^1(\R^2) \cap L^\infty(\R^2))$ and initial data $u_0$.
\item DiPerna and Majda \cite{DiPerna1987}: If $\omega(u_0) \in L^1(\R^2) \cap L^p(\R^2)$, then there exists a weak solution $u \in C([0,T];H_{loc}) \cap L^\infty(0,T;W^{1,p}_{loc}(\R^2) \cap E_m)$ with vorticity $\omega(u) \in L^\infty(0,T;L^1(\R^2) \cap L^p(\R^2))$ and initial data $u_0$.
\item Delort \cite{Delort1991}: If $\omega(u_0) \in \CalM^+_c(\R^2)$, then there exists $L > 1$ and a weak solution $u \in C([0,T];H^{-L}_{loc}(\R^2;\R^2)) \cap L^\infty(0,T;E_m)$ with vorticity $\omega(u) \in L^\infty(0,T;\CalM^+(\R^2))$ and initial data $u_0$. 
\item Vecchi and Wu \cite{Vecchi1993}: If $\omega(u_0) \in L^1_c(\R^2)$, then there exists $L > 1$ and a weak solution $u \in C([0,T];H^{-L}_{loc}(\R^2;\R^2)) \cap L^\infty(0,T;E_m)$ with vorticity $\omega(u) \in L^\infty(0,T;L^1(\R^2))$ and initial data $u_0$.
\end{enumerate}
\end{thm}

\begin{remark}\label{rem: E and H^-1}
Sometimes the condition of $\omega(u_0)$ being in $H^{-1}(\R^2)$ or $H^{-1}_{loc}(\R^2)$ is imposed. This is already included in our assumption $u_0 = m\Sigma + u_{0,kin} \in E_m$ in \Cref{thm: solutions of euler}: Indeed, $\omega(\Sigma)$ is an element of $C^\infty_c(\R^2) \subset H^{-1}(\R^2)$ due to the very definition of stationary vector fields. Moreover, from $u_{0,kin} \in H \subset L^2(\R^2;\R^2)$, it follows immediately that $\omega(u_{0,kin}) \in H^{-1}(\R^2)$.
\end{remark}

In case of the Navier-Stokes equations, no assumptions on the vorticity are required to obtain unique weak solutions.

\begin{thm}\label{thm: unique existence NSE}
Let $u_0 \in E_m$ for some $m \in \R$. Then there exists a unique weak solution $u \in C([0,T];E_m)$ of the Navier-Stokes equations with initial data $u_0$ and viscosity $\nu > 0$ so that $\nabla u \in L^2(0,T;L^2(\R^{2 \times 2}))$. Moreover, $u_{kin}$ satisfies the energy inequality
\begin{equation}\label{eq: energy inequ NSE}
\|u_{kin}\|_{C([0,T];L^2(\R^2))} + \nu\|\nabla u_{kin}\|_{L^2(0,T;L^2(\R^2))} \leq C\|u_{kin}(0)\|_{L^2(\R^2)},
\end{equation}
where $C$ is a constant depending exponentially on $T$ and $m$ (see also \eqref{eq: energy inequ kinetic part}). 
\end{thm}

\begin{lem}\label{lem: Biot-Savart}
Let $u \in \mathbb{E}$ be a vector field having a radial-energy decomposition with vorticity $\omega(u) \in \CalM(\R^2)$. Then $u$ is given by the Biot-Savart law, that is
\[u = K * \omega(u),\]
where $K(x) = \frac{1}{2\pi}\frac{x^\perp}{|x|^2}, x \in \R^2$. In particular, vector fields in $\mathbb{E}$ are uniquely determined by their vorticity.
\end{lem}

This can be proved based on the observation that a vector field whose coefficients are tempered distributions is uniquely determined by its divergence and vorticity, up to a vector field whose coefficients are harmonic polynomials (see \cite{Chemin1998}[Proposition 1.3.1]).\\
Let us also recall the following theorem on classical solutions (see e.g. \cite{Chemin1998}[Theorem 4.2.4] in case of the Euler equations):

\begin{thm}\label{thm: classic sols Euler}
Let $m \in \R$ and $r \geq 2$ be a natural number. For every $u_0 \in E_m \cap C^r(\R^2;\R^2)$ there exist unique solutions in $C([0,T];E_m) \cap L^\infty(0,T;C^r(\R^2;\R^2))$ of the Euler or Navier-Stokes equations with initial data $u_0$.
\end{thm}

One can even show that the weak solution of the Navier-Stokes equations becomes instantaneously smooth.\\
In general, for two smooth solutions $u_1, u_2$ of the Euler or Navier-Stokes equations with radial-energy decompositions 
\[u_1 = \sigma_1 + v_1, u_2 = \sigma_2 + v_2,\]
where $\sigma_1,\sigma_2$ are stationary solutions and $v_1 := u_1-\sigma_1,v_2 := u_2-\sigma_2 \in H$ are vector fields of finite kinetic energy, we have a relative energy inequality for the finite kinetic energy parts (see \cite{Bertozzi2001}[Proposition 3.4])
\begin{equation}\label{eq: stab est smooth sols}
\begin{split}
&\sup_{0\leq t \leq T}\|v_1 - v_2\|_{L^2(\R^2)}\\
\leq& \exp\left( \int_0^T \|\nabla v_2\|_{L^\infty(\R^2)} + \|\nabla\sigma_1\|_{L^\infty(\R^2)}\,dt \right)\bigg(\|v_1(0)-v_2(0)\|_{L^2(\R^2)}\\
&\quad + \int_0^T \|\sigma_1 - \sigma_2\|_{L^\infty(\R^2)}\|\nabla v_2(t)\|_{L^2(\R^2)} + \|\nabla\sigma_1 - \nabla\sigma_2\|_{L^\infty(\R^2)}\|v_2(t)\|_{L^2(\R^2)}\,dt\bigg),
\end{split}
\end{equation}
which can be proved using the Gronwall inequality. For this, in \cite{Bertozzi2001}, the estimate
\[\int_{\R^2} ((v_1-v_2)\cdot\nabla v_2)\cdot (v_1-v_2)\,dx \leq \|v_1-v_2\|_{L^2(\R^2)}^2\|\nabla v_2\|_{L^\infty(\R^2)}\]
was used. In the Navier-Stokes case, one can use Ladyzhenskaya's inequality to estimate
\begin{align}\label{eq: NSE Ladyzhenskaya}
&\int_{\R^2} ((v_1-v_2)\cdot\nabla v_2)\cdot (v_1-v_2)\,dx\\ \leq& \|v_1-v_2\|_{L^4(\R^2)}^2\|\nabla v_2\|_{L^2(\R^2)}\\ \leq& C\|v_1-v_2\|_{L^2(\R^2)}\|\nabla(v_1-v_2)\|_{L^2(\R^2)}\|\nabla v_2\|_{L^2(\R^2)}\\
\leq& C\|v_1-v_2\|_{L^2(\R^2)}^2\|\nabla v_2\|_{L^2(\R^2)}^2 + \frac{\nu}{2}\|\nabla(v_1-v_2)\|_{L^2(\R^2)}^2
\end{align}
for a constant $C$ depending on $\nu$, which we altered between the last and second last inequality, so that the term $\exp\left( \int_0^T \|\nabla v_2\|_{L^\infty(\R^2)} + \|\nabla\sigma_1\|_{L^\infty(\R^2)}\,dt \right)$ on the right-hand side \eqref{eq: stab est smooth sols} can be replaced by $\exp\left( C\int_0^T \|\nabla v_2\|_{L^2(\R^2)}^2 + \|\nabla\sigma_1\|_{L^\infty(\R^2)}\,dt \right)$. In case of the Navier-Stokes equations, one also obtains the gradient control
\begin{align}\label{eq: stab est grad smooth sols}
&\nu \int_0^T \|\nabla(v_1-v_2)\|_{L^2(\R^2)}^2\,dt\\
\leq& C\int_0^T \|v_1-v_2\|_{L^2(\R^2)}^2(\|\nabla v_2\|_{L^2(\R^2)}^2+ \|\nabla \sigma_1\|_{L^\infty(\R^2)})\\
&\qquad + \|\nabla(\sigma_1 - \sigma_2)\|_{L^\infty(\R^2)}\|v_2\|_{L^2(\R^2)}+\|\sigma_1 - \sigma_2\|_{L^\infty(\R^2)}\|\nabla v_2\|_{L^2(\R^2)}\,dt.
\end{align}
In particular, letting $u_2 = 0$ yields in \eqref{eq: stab est smooth sols} for both equations the energy inequality 
\begin{equation}\label{eq: energy inequ kinetic part}
\sup_{0\leq t \leq T}\|v_1\|_{L^2(\R^2)} \leq \exp\left(T \|\nabla\sigma_1\|_{L^\infty(\R^2)} \right)\|v_1(0)\|_{L^2(\R^2)}
\end{equation}
for the finite kinetic energy part and, in combination with \eqref{eq: stab est grad smooth sols}, specifically for the Navier-Stokes equations, the energy inequality \eqref{eq: energy inequ NSE}.\\
We also see from \eqref{eq: stab est smooth sols}, using the variation \eqref{eq: NSE Ladyzhenskaya}, whenever a sequence of weak solutions of the Navier-Stokes equations in $C([0,T];\mathbb{E})$ has bounded gradients in $L^2(0,T;L^2(\R^{2\times 2}))$ and whose initial data is Cauchy in $\mathbb{E}$, then that sequence is Cauchy in $C([0,T];\mathbb{E})$. We are going to use this fact later on to prove closedness in $C([0,T];\mathbb{E})$ of the set of weak solutions of the Navier-Stokes equations with fixed viscosity.\\
\hfill\\
We now fix $u_0 \in \mathbb{E}$ and suppose that $\omega(u_0)$ is in one of the classes described in \Cref{thm: solutions of euler}. Then a weak solution $u$ of the Euler or Navier-Stokes equations may be constructed by smoothing of the initial data and using the corresponding smooth solutions of the Euler or Navier-Stokes equations as approximating sequences. The weak solution $u$ inherits several properties and a priori estimates from the smooth approximations, which we will list now and assume in the following sections. These are by no means granted. For instance in \cite{Szekelyhidi2011}, weak solutions with vortex-sheet initial data (case $C)$ in \Cref{thm: solutions of euler}) are constructed, using convex integration methods, whose vorticity will in general no longer be a bounded measure in positive time.\\
From \eqref{eq: energy inequ kinetic part}, we derive
\begin{equation}\label{eq: energy ineq finite kin part}
\esssup_{0 \leq t \leq T}\|u_{kin}\|_{L^2(\R^2)} \leq \exp\left(T|m(u_0)| \|\nabla\Sigma\|_{L^\infty(\R^2)} \right)\|u_{0,kin}\|_{L^2(\R^2)}.
\end{equation} 
We introduce the constant
\begin{equation}\label{eq: constant a}
a := \exp(T\|\nabla\Sigma\|_{L^\infty(\R^2)}),
\end{equation}
whose dependency on $T$ and $\Sigma$ we omit in our notation as both are fixed throughout. Then
\begin{align}\label{eq: comp bdd L2_loc weak sols}
\|u(t)\|_{L^2(B_r(x_0))} &\leq \|m(u_0)\Sigma\|_{L^2(B_r(x_0))} + \|u_{kin}(t)\|_{L^2(B_r(x_0))}\\
&\leq |m(u_0)|(\|\Sigma\|_{L^\infty(\R^2)}\sqrt{\pi} r) + a^{|m(u_0)|}\|u_{0,kin}\|_{L^2(\R^2)}\\
&\leq \left(\frac{\|\Sigma\|_{L^\infty(\R^2)}\sqrt{\pi} r}{T\|\nabla\Sigma\|_{L^\infty(\R^2)}} + \|u_{0,kin}\|_{L^2(\R^2)} \right)a^{|m(u_0)|}\\
&\leq C\max\lbrace 1,r\rbrace(1+\|u_{0,kin}\|_{L^2(\R^2)})a^{|m(u_0)|}\\
&= C\max\lbrace 1,r\rbrace\gamma(u_0)
\end{align}
for any $r > 0, x_0 \in \R^2$ and almost every $0 \leq t \leq T$, where $C$ is a constant depending on $T$ and $\Sigma$ and $\gamma(u_0) := (1+\|u_{0,kin}\|_{L^2(\R^2)})a^{|m(u_0)|}$. We remark here that $\gamma$ is continuous with respect to $\|\cdot\|_{\mathbb{E}}$, depends exponentially on $m(u_0)$ and, in a sense, linearly on $u_{0,kin}$.\\
There exists $L \in \N$ and a constant $C = C(L)$ such that
\begin{equation}\label{eq: a priori est time derivative}
\esssup_{0\leq t \leq T} \|\partial_t u\|_{H^{-L}(\R^2)} \leq
\int_{\R^2} |\omega(u_0)|,
\end{equation}
where the term on the right-hand side denotes the total variation norm of $\omega(u_0)$. The proof of this in \cite{Hounie1999}[Theorem 2.1] works for both equations. This will usually guarantee equicontinuity of the velocity in $H^{-L}(\R^2;\R^2)$.\\
We may also infer continuity of the vorticity $\omega(u)$ for both equations depending on the class of initial vorticity that is considered. If 
\[\begin{cases}
\omega(u_0) \in (L^1\cap L^\infty)(\R^2), \text{ then } \omega(u) \in C([0,T];L^q(\R^2)) \text{ for every }1 \leq q < \infty\\
\omega(u_0) \in (L^1\cap L^p)(\R^2)\text{ for some }1 < p < \infty, \text{ then } \omega(u) \in C([0,T];L^p(\R^2))\\
\omega(u_0) \in L_c^1(\R^2)$ or $\omega(u_0) \in \CalM_c^+(\R^2), \text{ then }\omega(u) \in C_w([0,T];\CalM(\R^2)).
\end{cases}\]
The last condition means that for any $\psi \in C_0(\R^2)$, $t \mapsto \int_{\R^2}\psi\,d\omega(u)(t)$ is continuous.\\
We are also interested in conservation of the $L^p(\R^2)$ norm of the vorticity for the Euler equations. For this, we consider weak solutions whose vorticity $\omega(u)$ is a \textit{renormalized solution} of the vorticity equation of the Euler equations, i.e. in the sense of distributions 
\begin{equation}\label{eq: renormalized vorticity equation}
\partial_t\beta(\omega) + u \cdot \nabla\beta(\omega) = 0
\end{equation}
is satisfied for all $\beta \in C^1(\R^2)$ that are bounded and vanishing near $0$ with $u = K * \omega$ given by the Biot-Savart law (see for instance \cite{DiPerna1989} and \cite{LopesFilho2006} for more details).\\
At least for vorticity in $L^p(\R^2)$, $p > 1$, we may smoothly implement this into the presented work here. We remark though that \cite{Crippa2017} suggests that this also works in the $L^1(\R^2)$ case with some adaptations.\\
For $p \geq 2$, the vorticity of every weak solution is already renormalized \cite{LopesFilho2006}. For $1 < p < 2$, the vorticities of weak solutions obtained by (smooth) approximations of the Euler or Navier-Stokes equations in the inviscid limit are also renormalized (\cite{LopesFilho2006},\cite{Crippa2015}) and convergence holds in the strong sense (\cite{Ciampa2021}, \cite{NussenzveigLopes2021}).

\begin{lem}\label{lem: stab renormalized sols}
Suppose that $(\omega^n)_{n\in\N} \subset C([0,T];L^p(\R^2))$, $1 < p < \infty$, is a sequence of renormalized solutions of the vorticity formulation of the Euler equations or a sequence of vorticities of solutions $(u^n = K*\omega^n)_{n\in\N}$ of the Navier-Stokes equations with viscosity $\nu^n \to 0\,(n\to \infty)$.
If $\omega^n(0) \to \omega_0$ in $L^p(\R^2)$, then a subsequence of $(\omega^n)_{n\in\N}$ converges in $C([0,T];L^p(\R^2))$ to a renormalized solution of the Euler equations with initial data $\omega_0$.
\end{lem}

Renormalized solutions conserve all $L^q(\R^2)$ norms, $1 \leq q \leq \infty$, of the vorticity, i.e.
\begin{equation}\label{eq: a priori est vort}
\|\omega(u)(t)\|_{L^q(\R^2)} = \|\omega(u_0)\|_{L^q(\R^2)}
\end{equation}
for every $0 \leq t \leq T$ (with both sides possibly being $\infty$). In case of the Navier-Stokes equations, these quantities may at least not increase, i.e. \eqref{eq: a priori est vort} holds by replacing $=$ with $\leq$. Moreover, for solutions with initial vorticity in $L^1_c(\R^2)$ or $\CalM_c^+(\R^2)$, i.e cases $C)$ and $D)$ in \Cref{thm: solutions of euler}, we have
\begin{equation}\label{eq: a priori est vort L1,VS}
\begin{cases}
\int_{\R^2}\omega(u)(t) \leq \int_{\R^2}\omega(u_0) &: C)\\
\int_{\R^2}|\omega(u)(t)|\,dx \leq \int_{\R^2}|\omega(u_0)|\,dx &: D)
\end{cases}
\end{equation}
for all $0 \leq t \leq T$ in case of both equations.\\
For the initial data considered in case $B)$, i.e.\ when $\omega(u_0) \in L^1(\R^2) \cap L^p(\R^2)$ for some $1 < p < \infty$, we have
\[\esssup_{0 \leq t \leq T}\|\nabla u(t)\|_{L^p(\R^2)} \leq C(p)\|\omega(u_0)\|_{L^p(\R^2)},\]
which comes from a Calder\'{o}n-Zygmund argument.
Splitting the Biot-Savart kernel into a sum in $L^1(\R^2) + L^\infty(\R^2)$, Young's convolution inequality shows that $u = K * \omega(u) \in L^p(\R^2) + L^\infty(\R^2) \hookrightarrow L^p_{loc}(\R^2)$. Hence, for any ball $B_r(x_0) \subset \R^2$ of radius $r > 0$, there exists a constant $C = C(r,p)$ so that
\[\|u\|_{L^p(B_r(x_0))} \leq C(r,p)\|\omega(u_0)\|_{L^1(\R^2)\cap L^p(\R^2)}\]
and consequently also 
\begin{equation}\label{eq: a priori W^1,p}
\esssup_{0 \leq t \leq T}\|u\|_{W^{1,p}(B_r(x_0))} \leq C(r,p)\|\omega(u_0)\|_{L^1(\R^2)\cap L^p(\R^2)}.
\end{equation}
For the case of $\omega(u_0)$ just being an element of $\CalM_c^+(\R^2)$ or $L^1_c(\R^2)$, we will later use the following theorem, which may be derived from Theorem 6.3.1 in \cite{Chemin1998}. A measure $\mu \in \CalM(\R^2)$ is called continuous if it has no atoms, i.e.\ $\mu(\lbrace x \rbrace) = 0$ for every $x \in \R^2$. By Lemma 6.3.2 in \cite{Chemin1998}, $\mu$ being an element of $H^{-1}(\R^2)$ is a sufficient condition for $\mu$ to be continuous. Compare this to \Cref{rem: E and H^-1}.\\
We also recall the Hahn decomposition $\mu = \mu^+ - \mu^-$, where $\mu^+$ and $\mu^-$ are positive, finite measures in $\CalM^+(\R^2)$. Then $|\mu| := \mu^+ + \mu^-$ denotes the associated total variation measure.

\begin{thm}\label{thm: M L^1 limit result}
Let $(u^n)_{n\in\N} \subset L^\infty(0,T;\mathbb{E})$ be a sequence of weak solutions of the Euler or Navier-Stokes equations, converging weakly-* to some $u \in L^\infty(0,T;\mathbb{E})$ and suppose that the vorticities $(\omega(u^n))_{n\in\N}$ converge weakly-* to $\omega(u)$ in $L^\infty(0,T;\CalM(\R^2))$. If $(|\omega(u^n)|)_{n\in\N}$ converges weakly-* to some $\omega^+$ in $L^\infty(0,T;\CalM(\R^2))$, where $\omega^+(t)$ is continuous for almost every $0 \leq t \leq T$, then $u$ is a weak solution of the Euler equations.
\end{thm} 

Later on, when we consider case $C)$ in \Cref{thm: solutions of euler} of the initial vorticity being in $\CalM^+_c(\R^2)$, this will be very easy to check since for any $\mu \in \CalM^+(\R^2)$, $|\mu| = \mu$.\\
In case $D)$ in \Cref{thm: classic sols Euler}, some more care is required. The following property that we may demand for the weak solution $u$ of the Euler or Navier-Stokes equations, which has also been the key in the original arguments by Vecchi and Wu in \cite{Vecchi1993}, will be of great help showing the non-concentration assumptions of \Cref{thm: M L^1 limit result}:
\begin{equation}\label{eq: local conservation vorticity}
\sup_{0 \leq t \leq T}\int_{B_r(x_0)} |\omega(u)|\,dx \leq \sup_{|E| = \pi r^2}\int_E |\omega(u_0)|\,dx,
\end{equation}
for all $r > 0$ and $x_0 \in \R^2$, where on the right-hand side, we consider the supremum over the set of all Borel measurable sets $E$ in $\R^2$ whose Lebesgue measure $|E|$ is equal to $\pi r^2$. For smooth solutions of the Euler equations, this follows from the fact that the vorticity is being transported by the velocity. As pointed out in the proof of Theorem 3.5 in \cite{Schochet1995}, in case of the Navier-Stokes equations, one may derive this by constructing the Navier-Stokes solution via viscous-splitting, i.e.\ approximation of iteratively solving the Euler and heat equation over time scales of vanishing size (see the discussion on p.\ 1089 in \cite{Schochet1995} and \cite{Majda1993}[Section 3.4] for an introduction to viscous splitting), both of which leave the estimate \eqref{eq: local conservation vorticity} intact.

\subsection{Preliminaries on the abstract framework of statistical solutions}\label{sec: abstract framework}

The Borel-$\sigma$-algebra on a topological space $X$ will be denoted by $\CalB(X)$. For an interval $I\subset\R$, the Lebesgue-$\sigma$-algebra on $I$ will be denoted by $\CalL(I)$. Then the space of continuous functions between $I$ and $X$, endowed with the compact-open topology, will be denoted by $C_{loc}(I;X)$.

\begin{defi}\label{def:traj stats sol}
Let $X$ be a Hausdorff space, $I \subset \R$ some arbitrary interval and let $\CalU \subset \CalX := C_{loc}(I;X)$. A Borel probability measure $\rho$ on $\CalX$ is called a $\CalU$-trajectory statistical solution (over $I$) if
\begin{enumerate}[i)]
\item $\rho$ is inner regular, i.e.\
\[\rho(A) = \sup_{\substack{{K \subset A}\\{K \text{compact in }\CalX}}} \rho(K)\]
for every Borel measurable subset $A$ of $\CalX$;
\item there exists $\CalV \subset \CalU$ which is Borel measurable in $\CalX$ such that $\rho(\CalV) = 1$. In this case, we say that $\rho$ is carried by $\CalU$. 
\end{enumerate} 
\end{defi}

\begin{remark}\label{rem: def traj stats sol}
\begin{enumerate}[i)]
\item The set $\CalU$ can be thought of as the set of solutions of a partial differential equation with phase space $X$. Also, the carrier $\CalU$ does not necessarily have to Borel measurable.
\item The assumption of inner regularity is owed to the generality in which statistical solutions are discussed in \cite{Bronzi2016}. There, a special topology on the space of finite Borel measures $\CalM(X)$ is used, introduced by F. Topsoe in \cite{Topsoe1970}, which the authors label weak-* semicontinuity topology. This topology turns the subspace of inner regular measures in $\CalM(X)$ into a Hausdorff space. For a completely regular Hausdorff space $X$, this topology coincides with the standard weak-* topology on $\CalM(X)$.\\
Moreover, on Polish spaces, i.e.\ completely metrizable spaces, every finite Borel measure is already inner regular (\cite{Parthasarathy1967}[Theorem 3.2]). Most of the spaces considered here are Polish spaces and consequently, we will usually ignore condition i) in \Cref{def:traj stats sol}.
\end{enumerate}
\end{remark}

Let us also recall here the precise definition of the support of a Borel measure. Let $X$ be a topological space and for every $x \in X$, we denote by $\mathcal{O}_x$ the set of all open neighbourhoods of $x$. For a measure $\mu$ on the Borel-$\sigma$-algebra of $X$, the support of $\mu$ is defined as
\[\supp\mu = \lbrace x \in X: \forall O \in \CalO_x: \mu(O) > 0\rbrace.\] 
The support of a measure is closed. Indeed, its complement is the union of all open sets of measure $0$.

One can also consider time parametrized measures. The notations used in the following definition are explained below.

\begin{defi}\label{def: phase space stats sol}
Let $X,Z$ be Hausdorff spaces and let $Y$ be a topological vector space such that $Z \subset X \subset Y_{w*}'$ with continuous injections. Let $I \subset \R$ be an arbitrary interval and consider a function $F: I \times Z \to Y'$.\\
A family $\lbrace\rho_t\rbrace_{t\in I}$ of Borel probability measures on $X$ is called a statistical solution (in phase space) of the evolution equation $u_t = F(t,u)$, over the interval $I$, if
\begin{enumerate}[i)]
\item the mapping 
\[t \mapsto \int_X \varphi(u)\,d\rho_t(u)\]
belongs to $C_b(I)$, the space of bounded continuous functions on $I$, for every $\varphi \in C_b(X)$;
\item for almost every $t \in I$, $\rho_t$ is carried by $Z$ and $u \mapsto \langle F(t,u),v\rangle_{Y',Y}$ is $\rho_t$-integrable for every $v \in Y$;
\item the mapping 
\[t \mapsto \int_X \langle F(t,u),v\rangle_{Y',Y}\,d\rho_t(u)\]
belongs to $L^1_{loc}(I)$ for every $v \in Y$;
\item for any cylindrical test function $\Phi \in Y'$, we have
\begin{equation}\label{eq: Foias-Liouville equ}
\int_X \Phi(u)\,d\rho_t(u) = \int_X \Phi(u)\,d\rho_{t'}(u) + \int_{t'}^t\int_X \langle F(s,u),\Phi'(u)\rangle_{Y',Y}\,d\rho_s(u)\,ds
\end{equation}
for all $t,t' \in I$.
\end{enumerate}
\end{defi}

Here, $\langle \cdot,\cdot\rangle_{Y',Y}$ denotes the duality product between $Y$ and its topological dual $Y'$. Moreover, $Y'_{w*}$ specifically denotes $Y'$ endowed with the weak-* topology.\\ 
A function $\Phi: Y' \to \R$ is called cylindrical test function in $Y'$ if there exists $\phi \in C_c^1(\R^k)$ for some $k \in \N$ and $v_1,...,v_k\in Y$ such that
\[\Phi(u) = \phi(\langle u,v_1 \rangle_{Y',Y},...,\langle u,v_k \rangle_{Y',Y})\]
for all $u \in Y'$.\\
For such a cylindrical test function $\Phi$ and $w \in Y'$, one may compute the G\^{a}teaux derivative $\Phi'(u)$ of $\Phi$ at $u \in Y'$ in direction of $w \in Y'$ by
\begin{equation}\label{eq: derivative cylinder fct}
\begin{split}
\langle w, \Phi'(u)\rangle_{Y',Y''} &= \sum_{j=1}^k\partial_j\phi(\langle u,v_1 \rangle_{Y',Y},...,\langle u,v_k \rangle_{Y',Y})\langle w,v_j\rangle_{Y',Y}\\
&= \langle w, \sum_{j=1}^k\partial_j\phi(\langle u,v_1 \rangle_{Y',Y},...,\langle u,v_k \rangle_{Y',Y}) v_j\rangle_{Y',Y}.
\end{split}
\end{equation}
Consequently, we identify $\Phi'(u)$ with $\sum_{j=1}^k\partial_j\phi(\langle u,v_1 \rangle_{Y',Y},...,\langle u,v_k \rangle_{Y',Y})v_j$ as an element in $Y$. Moreover, we see from this representation that $\Phi'$ is continuous on $Y'_{w*}$.\\

In \cite{Bronzi2016}, we are provided with existence results for both trajectory statistical solutions and statistical solutions in phase space. 

\begin{thm}\label{thm: ex traj stats sol}
Let $X$ be a Hausdorff space and let $I$ be some arbitrary interval, which is closed and bounded on the left with endpoint $t_0$. Let $\CalU \subset \CalX := C_{loc}(I;X)$ and suppose
\newcounter{hypthTrajStatsSol}
\begin{enumerate}[({H}1)]
\item $\Pi_{t_0}\CalU =X$, where $\Pi_{t_0}:\CalX \to X, u \mapsto u(t_0)$ is the time evaluation mapping at the initial time $t_0$;
\setcounter{hypthTrajStatsSol}{\value{enumi}}
\end{enumerate}
there exists a family $\CalK'(X)$ of compact subsets of $X$ such that 
\begin{enumerate}[({H}1)]
\setcounter{enumi}{\value{hypthTrajStatsSol}}
\item every inner regular Borel probability measure $\mu_0$ on $X$ (see \Cref{def:traj stats sol}) is inner regular with respect to $\CalK'(X)$, i.e.\ $\sup_{\substack{{K \subset A}\\{K \in \CalK(X)}}} \mu_0(K)$ for every Borel measurable set $A$ in $X$;
\item for every $K \in \CalK'(X)$, the set $\Pi_{t_0}^{-1}(K) \cap \CalU$ is compact in $\CalX$.
\end{enumerate}
Then, for any inner regular Borel probability measure $\mu_0$ on $X$, there exists a $\CalU$-trajectory statistical solution $\rho$ over $I$ such that $\Pi_{t_0}\rho = \mu_0$, i.e.\ $\rho(\Pi_{t_0}^{-1}(A)) = \mu_0(A)$ for every Borel measurable subset $A$ of $X$.
\end{thm}

\begin{remark}
\begin{enumerate}[i)]
\item (H1) means that for each initial value in $X$, we may find a trajectory in $\CalU$ to the corresponding initial value problem on $I$.
\item (H2) allows one to pull back to the case of inner regular initial distributions of compact support as arbitrary inner regular initial distributions may be approximated through exhaustion of their support by sets in $\CalK'(X)$. 
\item (H3) states that the set of solution trajectories with initial data in $K \in \CalK'(X)$ is compact. This can usually be shown by adapting the compactness proofs that yield existence of solutions in the first place and will also be our strategy later.
\end{enumerate}
\end{remark}

In case that the conditions of \Cref{thm: ex traj stats sol} only hold for a subset of inner regular Borel probability measures or, in a related way, continuity of the solutions only holds in a coarser topology, there is another version of this existence theorem available, which is going to be the content of the following theorem.

\begin{thm}\label{thm: ex traj stats sol v2}
Let $X$ be a Hausdorff space and let $I$ be some arbitrary interval, which is closed and bounded on the left with endpoint $t_0$. Let $\CalU \subset \CalX := C_{loc}(I;X)$ and let $X_0$ be a Borel subset of $X$. Suppose that 
\begin{enumerate}[({H}1')]
\item $\Pi_{t_0}\CalU \supset X_0$;
\setcounter{hypthTrajStatsSol}{\value{enumi}}
\end{enumerate}
there exists a family $\CalK'(X_0)$ of compact subsets of $X_0$ such that 
\begin{enumerate}[({H}1')]
\setcounter{enumi}{\value{hypthTrajStatsSol}}
\item every inner regular Borel probability measure $\mu_0$ on $X$ (see \Cref{def:traj stats sol}), which is carried by $X_0$, is inner regular with respect to $\CalK'(X_0)$, i.e.\ $\mu_0(A) = \sup_{\substack{{K \subset A}\\{K \in \CalK'(X_0)}}} \mu_0(K)$ for every Borel measurable $A \subset X$;
\item for every $K \in \CalK'(X_0)$, the set $\Pi_{t_0}^{-1}(K) \cap \CalU$ is compact in $\CalX$.
\end{enumerate}
Then, for any inner regular Borel probability measure $\mu_0$ on $X$, which is carried by $X_0$, there exists a $\CalU$-trajectory statistical solution $\rho$ over $I$ such that $\Pi_{t_0}\rho = \mu_0$,
\end{thm}

We are going to use this theorem later in the first way described right before \Cref{thm: ex traj stats sol v2}. $X$ will be a very general space, whose natural topology is the one in which we have continuity of our solutions, e.g. $H^{-L}_{loc}(\R^2;\R^2)$ (see \Cref{def: Euler weak vel}), whereas $X_0$ incorporates the assumptions on the initial velocity and vorticity that are required to obtain solutions in each specific class, so that \textit{(H1')} is satisfied.\\
Different choices of $X$ or $X_0$ are absolutely possible though. Let us also point here towards Remark 3.7 in \cite{Bronzi2016}: In case of the Navier-Stokes equations on  a bounded domain, one could use \Cref{thm: ex traj stats sol} with the choice of $X = H_w$, i.e.\ $H$ endowed with the weak topology, as weak Leray-Hopf solutions are continuous with respect to this topology. Alternatively, one could apply \Cref{thm: ex traj stats sol v2} with the choice of $X_0 = H_w$ and $X = D(A^{-1/2})$, i.e.\ viewing Leray-Hopf solutions as continuous functions with values in the domain of sufficiently large negative powers of the Stokes operator.\\
\hfill\newline
For the existence result of statistical solutions in phase space, let us give some explanation beforehand in which sense the evolution equation $u_t = F(t,u)$ is to be understood.\\
As before, let $I \subset \R$ be some arbitrary interval, $Y$ a topological vector space and $X,Z$ Hausdorff spaces such that $Z \subset X \subset Y_{w*}'$ with continuous injections. As in \cite{Bronzi2016}, we say that $u: I \to Y'$ has a certain property $P$ weak-* scalarwise if for every $v \in Y$, the function $I \to \R, t \mapsto \langle u(t),v \rangle_{Y',Y}$, has the property $P$. Then we define $\CalZ$ to be the space of functions $u: I \to X$ such that $u(t) \in Z$ for almost every $t \in I$. Let $\CalY_1$ be the space of functions $u \in \CalX := C_{loc}(I;X)$ that are weak-* scalarwise absolutely continuous such that there exists $w: I \to X$ which is weak-* scalarwise locally integrable and satisfies $\frac{d}{dt} \langle u(t),v\rangle_{Y',Y} = \langle w(t),v\rangle_{Y',Y}$
for almost every $t \in I$ and for all $v \in Y$.

\begin{thm}\label{thm: phase space stats sol}
Let $X,Z$ be Hausdorff spaces and $Y$ be a topological vector space such that $Z \subset X \subset Y_{w*}'$ with continuous injections. Let $I \subset \R$ be an arbitrary interval and consider a function $F: I \times Z \to Y'$ as well as $\CalU \subset \CalX := C_{loc}(I;X)$. Suppose that $\rho$ is a $\CalU$-statistical solution with $\rho(\CalV) = 1$ for some $\CalV \subset \CalU$, which is Borel measurable in $\CalX$. We assume
\begin{enumerate}[i)]
\item $\CalB(Z) \subset \CalB(X)$;
\item $\CalU \subset \CalX_1 := \CalZ \cap \CalY_1$;
\item $F: I \times Z \to Y'$ is $(\CalL(I) \otimes \CalB(Z))$-$\CalB(Y')$ measurable such that
\begin{itemize}
\item $t \mapsto F(t,u(t))$ is weak-* scalarwise locally integrable on $I$ for $u \in \CalX_1$;
\item $u_t = F(t,u)$ in the weak sense for every $u \in \CalU$, i.e.\
\[\frac{d}{dt}\langle u(t),v\rangle_{Y',Y} = \langle F(t,u(t)),v\rangle_{Y',Y}\]
in the sense of distributions on $I$, for all $v \in Y$;
\end{itemize}
\item the mapping 
\[t \mapsto \int_{\CalV} |\langle F(t,u(t)),v\rangle_{Y',Y}|\,d\rho(u)\]
belongs to $L_{loc}^1(I)$ for every $v \in Y$.
\end{enumerate}
Then 
\[t \mapsto \int_\CalV \varphi(u(t))\,d\rho(u)\in C_b(I)\]
for every $\varphi \in C_b(X)$ and for every cylinder functional $\Phi$ in $Y'$ and $t',t \in I$, we have
\[\int_{\CalV} \Phi(u(t))\,d\rho(u) = \int_{\CalV} \Phi(u(t'))\,d\rho(u) + \int_{t'}^t\int_{\CalV} \langle F(s,u(s)),\Phi'(u(s))\rangle_{Y',Y}\,d\rho(u).\]
In particular, the family of projections $\lbrace \rho_t \rbrace_{t\in I}$, $\rho_t = \Pi_t\rho$, is a statistical solution in phase space of the equation $u_t = F(t,u)$ over $I$.  
\end{thm}

The statistical solutions in phase space as constructed in \Cref{thm: phase space stats sol} by projecting trajectory statistical solutions in time will also be called \emph{projected statistical solutions in phase space}. It may not always be clear whether every statistical solution in phase space is already a projected statistical solution in phase space.\\ 
In combination with \Cref{thm: ex traj stats sol} and \Cref{thm: ex traj stats sol v2}, we obtain the following results.

\begin{thm}\label{thm: ex phase space stats sol}
Let $X$ and $Z$ be Hausdorff spaces and $Y$ be a topological vector space such that $Z \subset X \subset Y'_{w*}$ with continuous injections. Let $I \subset \R$ be an arbitrary interval, which is closed and bounded on the left with left endpoint $t_0$, and let $\CalU \subset \CalX := C_{loc}(I;X)$ satisfy the hypotheses (H1), (H2) and (H3) of \Cref{thm: ex traj stats sol}. We assume
\begin{enumerate}[({H}1)]
\setcounter{enumi}{3}
\item $\CalB(Z) \subset \CalB(X)$;
\item $\CalU \subset \CalX_1 := \CalZ \cap \CalY_1$ (see the explanation above \Cref{thm: phase space stats sol}) and $F: I \times Z \to Y'$ is $(\CalL(I) \otimes \CalB(Z))$-$\CalB(Y')$ measurable such that
\begin{itemize}
\item $t \mapsto F(t,u(t))$ is weak-* scalarwise locally integrable on $I$ for $u \in \CalX_1$;
\item $u_t = F(t,u)$ in the weak sense for every $u \in \CalU$, i.e.\
\[\frac{d}{dt}\langle u(t),v\rangle_{Y',Y} = \langle F(t,u(t)),v\rangle_{Y',Y}\]
in the sense of distributions on $I$, for all $v \in Y$;
\end{itemize}
\item there exists a function $\gamma: I \times X \times Y \to \R$ such that for every $v \in Y$ the mapping $(t,u) \mapsto \gamma(t,u,v)$ is $\CalL(I)\otimes\CalB(X)$ measurable and
\[\int_{t_0}^t |\langle F(s,u(s)),v\rangle_{Y',Y}|\,ds \leq \gamma(t,u(t_0),v)\]
for all $t \in I, u \in \CalU$.
\end{enumerate} 
Then, for any inner regular Borel probability measure $\mu_0$ on $X$, satisfying
\begin{equation}\label{eq: finiteness moments phase space stats sol}
\int_X \gamma(t,u_0,v)\,d\mu_0(u_0) < \infty
\end{equation}
for almost every $t \in I$ and all $v \in Y$, there exists a projected statistical solution $\lbrace \rho_t \rbrace_{t \in I}$ such that $\rho_{t_0} = \mu_0$.
\end{thm}

\begin{thm}\label{thm: ex phase space stats sol v2}
Consider the framework of \Cref{thm: ex phase space stats sol} and let $X_0$ be a Borel subset of $X$. Suppose, however, that $\CalU$ satisfies the conditions (H1'), (H2') and (H3') in \Cref{thm: ex traj stats sol v2} instead of (H1), (H2) and (H3). If also (H4), (H5) and (H6) from \Cref{thm: ex phase space stats sol} hold, then for any inner regular Borel probability measure $\mu_0$ on $X$ which is carried by $X_0$ and satisfies \eqref{eq: finiteness moments phase space stats sol}, there exists a projected statistical solution $\lbrace \rho_t \rbrace_{t \in I}$ such that $\rho_{t_0} = \mu_0$.
\end{thm}

We remark that $\gamma$, as defined by us in \eqref{eq: comp bdd L2_loc weak sols}, has a slightly different form compared to $\gamma$ in \eqref{eq: finiteness moments phase space stats sol}, but it will play the same role. 

\section{Trajectory statistical solutions of the Euler equations}
In this section, we are going to construct trajectory statistical solutions of the Euler equations in each of the discussed classes of \Cref{sec: euler} and of the Navier-Stokes equations. We will begin by discussing trajectory statistical solutions in the class of weak solutions as in $A)$ in \Cref{thm: solutions of euler}, i.e.\ weak solutions of the Euler equations with uniformly bounded vorticity and in the class of weak solutions of the Navier-Stokes equations. These two cases are somewhat similar as we do have unique solutions. Moreover, we will not need the existence results from \Cref{sec: abstract framework} but instead construct the trajectory statistical solutions in a more straightforward way as pushforward measures of the initial distribution under the corresponding solution operator. Our only task here is to check the required measurability properties.\\
In the second part of this section, we consider the remaining classes of solutions $B)$ - $D)$ in \Cref{thm: solutions of euler}. Even though each case there has its own intricacies, all of them may be handled similarly by making use of the existence result \Cref{thm: ex traj stats sol v2}. We will in fact demonstrate the idea behind \Cref{thm: ex traj stats sol} and \Cref{thm: ex traj stats sol v2} in the particular case of $B)$, when the vorticity is in $L^p(\R^2)$, $1 < p < \infty$. 

\subsection{Vorticity in $L^\infty(\R^2)$ and the Navier-Stokes class}\label{sec: stats sol infty}\label{sec: traj sols infty and NS}

We consider the spaces 
\begin{itemize}
\item $X^\infty = X_{NS}^\nu = \mathbb{E}$;
\item $X_0^\infty = \lbrace u_0 \in \mathbb{E} : \omega(u_0) \in L^1(\R^2) \cap L^\infty(\R^2)\rbrace$;
\item $\CalU^\infty,$ the space of weak solutions of the Euler equations $u \in C([0,T];\mathbb{E})$ with vorticity $\omega(u) \in C([0,T];L^q(\R^2)) \cap L^\infty(0,T;L^1(\R^2)\cap L^\infty(\R^2))$, $1 \leq q < \infty$, and initial data in $X_0^\infty$;
\item $\CalU^\nu_{NS}$, the space of weak solutions of the Navier-Stokes equations $u \in C([0,T];\mathbb{E})$ such that $\nabla u \in L^2(0,T;L^2(\R^{2\times 2}))$ with initial data in $\mathbb{E}$ and viscosity $\nu$,  
\end{itemize}
where $\nu > 0$ will be a fixed constant in the following, unless stated otherwise.

\begin{remark}
\begin{enumerate}[i)]
\item We introduced $X^\infty = X^\nu_{NS} = \mathbb{E}$ here to indicate the connection to \Cref{def:traj stats sol}. However, hoping to keep things clearer, we will usually not use this notation but rather write out $\mathbb{E}$, so it is particularly clear which norm or topology is considered. 
\item Unlike in the next subsection, properties such as \eqref{eq: a priori est vort} need not be explicitly demanded in the definition of $\CalU^\infty$ since all weak solutions in $\CalU^\infty$ have vorticities that are renormalized solutions of the vorticity formulation of the Euler equations.
\end{enumerate} 
\end{remark}

There exists a solution operator $S^\infty: X_0^\infty \to \CalU^\infty \subset C([0,T];\mathbb{E})$ so that for any $u_0 \in X_0^\infty$, $S^\infty(u_0)$ is the unique weak solution of the Euler equations with initial data $u_0$ as given by \Cref{thm: solutions of euler}. Likewise, there exists an operator $S^\nu_{NS}: \mathbb{E} \to \CalU_{NS}^\nu \subset C([0,T];\mathbb{E})$ so that $S^\nu_{NS}(u_0)$ is the unique weak solution of the Navier-Stokes equations with initial data $u_0$ and viscosity $\nu$ as given by \Cref{thm: unique existence NSE}.\\
For instance in the latter case, given an initial distribution $\mu_0$, the natural way to describe the distribution $\rho_{NS}^\nu$ of the solutions in $\CalU^\nu_{NS}$ with respect to $\mu_0$ would be to define $\rho_{NS}^\nu$ as the pushforward measure $S_{NS}^\nu\mu_0$, that is $\rho_{NS}^\nu(A) = \mu_0((S_{NS}^\nu)^{-1}(A))$ for all Borel measurable sets $A$ in $C([0,T];\mathbb{E})$. This would immediately imply that $\rho_{NS}^\nu$ is carried by $\CalU^\nu_{NS}$ in the sense of \Cref{def:traj stats sol} if we could find a measurable subset of $\CalU^\nu_{NS}$ in $C([0,T];\mathbb{E})$ with measure $1$, or just show that $\CalU^\nu_{NS}$ is measurable itself. This definition, however, requires that $S^\nu_{NS}$ is Borel measurable.\\
We begin by proving the Borel measurability of $X_0^\infty$ in $\mathbb{E}$. 

\begin{lem}\label{lem: measurability X^infty}
The space $X_0^\infty$ is Borel measurable in $\mathbb{E}$.
\end{lem}

\begin{proof}
Letting $A_k = \lbrace u \in X_0^\infty : \|\omega(u)\|_{L^1(\R^2)}, \|\omega(u)\|_{L^\infty(\R^2)} \leq k\rbrace$ for all $k \in \N$, we may write
\[X_0^\infty = \bigcup_{k\in\N} A_k.\]
Therefore, to prove the lemma, it suffices to prove that $A_k$ is closed in $\mathbb{E}$. Let $(u^n)_{n\in\N}$ be a sequence in $A_k$ which converges in $\mathbb{E}$ to some $u$. Due to the definition of $A_k$ and weak-* compactness, we may also find subsequences for which
\begin{align}
&\omega(u^n) \overset{\ast}{\rightharpoonup} \omega(u)\,(n \to \infty) \text{ in } \CalM(\R^2),\\
&\omega(u^n) \overset{\ast}{\rightharpoonup} \omega(u)\,(n \to \infty) \text{ in } L^\infty(\R^2).
\end{align} 
The fact that on the right-hand side we could write $\omega(u)$ instead of some unspecific element in the according spaces follows from the convergence of $(u^n)_{n\in\N}$ in $\mathbb{E}$. Now with $\omega(u)$ being an element of $L^\infty(\R^2) \subset L^1_{loc}(\R^2)$, we obtain from weak-* convergence in the sense of measures that for any $r > 0$
\[\|\omega(u)\|_{L^1(B_r(0))} \leq \liminf_{n\to\infty} \|\omega(u^n)\|_{L^1(B_r(0))} \leq k,\]
so that $\|\omega(u)\|_{L^1(\R^2)} \leq k$. By also concluding from weak-* convergence in $L^\infty(\R^2)$ that $\|\omega(u)\|_{L^\infty(\R^2)} \leq k$, we obtain $u \in A_k$.
\end{proof}

Before proceeding with the following lemma, let us define the smoothing operators $(\CalJ^\varepsilon)_{\varepsilon > 0}$: Fix a non-negative $\eta \in C_c^\infty(\R^2)$ having support in the closed unit ball $\overline{B}_1(0)$ and satisfying $\int_{\R^2}\eta \,dx = 1$. Then define $\eta^{\varepsilon}(x) = \frac{1}{\varepsilon^2}\eta\left(\frac{x}{\varepsilon}\right), x \in \R^2,$ and finally let 
\begin{equation}\label{eq: smoothing op}
\CalJ^{\varepsilon} f := \eta^{\varepsilon} * f
\end{equation}
for all $f \in L^1_{loc}(\R^2)$.

\begin{lem}\label{lem: measurability S infty nu}
The operator $S^\infty$ is $\CalB(X_0^\infty)$-$\CalB(C([0,T];\mathbb{E}))$ measurable, the operator $S^\nu_{NS}$ is $\CalB(\mathbb{E})$-$\CalB(C([0,T];\mathbb{E}))$ measurable.
\end{lem}

\begin{proof}
We prove the stated measurability of $S^\infty$, where we begin by introducing the following operators (for the definition of $\Sigma$, see \eqref{eq: fixed stat vec field}):
\begin{itemize}
\item $\displaystyle G: X_0^\infty \to \R\times H,$ $u_0 \mapsto (m(u_0),u_{0,kin});$
\item $S^{\R\times H}: \R \times H \to \CalU^\infty, (m,v) \mapsto S^\infty(m\Sigma + v)$.
\end{itemize}
We are going to argue that $S^{\infty}$ can be written as the pointwise limit $(\varepsilon \to 0)$ of 
\[S^{\varepsilon,\infty}: X_0^\infty \to \CalU^\infty, S^{\varepsilon,\infty} = (S^{\R\times H} \circ (\Id_{\R},\CalJ^\varepsilon))\circ G, \varepsilon > 0,\]
which we now prove to be measurable as a composition of measurable functions.\\
The definition of $\|\cdot\|_{\mathbb{E}}$ immediately implies that $G$ is even continuous.\\
As for the measurability of $S^{\R\times H} \circ (\Id_{\R}, \CalJ^\varepsilon)$, we note that for any $(m,v) \in \R\times H$, the function $u = (S^{\R \times H} \circ (\Id_{\R},\CalJ^\varepsilon))(m,v)$ is the smooth solution of the Euler equations in $\CalU^\infty$ with smooth initial data $m\Sigma + \CalJ^\varepsilon v \in E_m$ (see \Cref{thm: classic sols Euler}).\\
The stability estimate \eqref{eq: stab est smooth sols} for smooth solutions of the Euler equations yields the continuity of $S^{\R\times H} \circ (\Id_{\R},\CalJ^\varepsilon)$ between the spaces $\R\times H$ and $C([0,T];\mathbb{E})$.\\
The pointwise convergence of $S^{\varepsilon,\infty}$ to $S^\infty$ is now a classic argument by which the weak solutions in the Yudovich class may actually be constructed. See for instance \cite{Chemin1998}[Chapter 5], where it is proved that for $m \in \R$ and any $u_0 \in E_m$, the solutions with smoothed initial data $(\CalJ^\varepsilon u_0)_{\varepsilon > 0}$ are Cauchy in $C([0,T];E_m)$. The fact that in the definition of $S^{\varepsilon,\infty}$ we only smoothed the initial finite kinetic energy part does not make a difference because $\|\CalJ^\varepsilon \Sigma - \Sigma\|_{L^2(\R^2)} \to 0\,(\varepsilon \to 0)$ (see \cite{Chemin1998}[Lemma 5.1.2]).\\
The measurability of $S^\nu_{NS}$ may be proved analogously as we chose $X^\nu_{NS} = X^\infty = \mathbb{E}$ and the important ingredient in the proof above being the stability estimate \eqref{eq: stab est smooth sols}, which also holds in the Navier-Stokes case. 
\end{proof}

\begin{lem}\label{lem: meas U infty nu}
The space $\CalU^\infty$ is Borel measurable in $C([0,T];\mathbb{E})$, the space $\CalU^\nu_{NS}$ is closed in $C([0,T];\mathbb{E})$. 
\end{lem}

\begin{proof}
We may write $\CalU^\infty = \bigcup_{k=1}^\infty A_k$ with $A_k = \lbrace u \in \CalU^\infty: \|\omega(u(0))\|_{L^1(\R^2)\cap L^\infty(\R^2)} \leq k\rbrace$ for every $k \in \N$. Then we show that each set $A_k$, $k \in \N$, is closed in $C([0,T];\mathbb{E})$. Fix $k \in \N$ and let $(u^n)_{n\in\N} \subset A_k$ be a sequence converging to some $u$ in $C([0,T];\mathbb{E})$. This also implies convergence in $C([0,T];H_{loc})$, which is strong enough to conclude that $u$ is a weak solution of the Euler equations with initial data $u(0)$. From the definition of $A_k$, we may derive similarly to \Cref{lem: measurability X^infty} that $u(0) \in X_0^\infty$ satisfies $\|\omega(u(0))\|_{L^1(\R^2) \cap L^\infty(\R^2)} \leq k$. Moreover, by \eqref{eq: a priori est vort}, $\|\omega(u)\|_{L^1(\R^2) \cap L^\infty(\R^2)}$ remains essentially bounded in time by $k$, so that $\omega(u) \in L^\infty(0,T;L^1(\R^2) \cap L^\infty(\R^2))$ and we conclude $u \in A_k$.\\
Now we show the closedness of $\CalU^\nu_{NS}$ in $C([0,T];\mathbb{E})$. Let $(u^n)_{n\in\N} \subset \CalU^\nu_{NS}$ be a sequence converging to some $u$ in $C([0,T];\mathbb{E})$. Due to \eqref{eq: energy inequ NSE}, $(\nabla u^n)_{n\in\N}$ is also bounded in $L^2(0,T;L^2(\R^{2 \times 2}))$, from which we may derive weak-* convergence of $(\nabla u^n)_{n\in\N}$ to $\nabla u$ in $L^2(0,T;L^2(\R^{2 \times 2}))$ so that $u$ is in the desired space of functions. Both of these types of convergence suffice to pass to the limit in each term in the weak formulation of the Navier-Stokes equations so that $u$ is once again a weak solution of the two-dimensional Navier-Stokes equations with viscosity $\nu$.
\end{proof}

As outlined at the beginning of this section, we now obtain an existence result for $\CalU^\infty$-trajectory statistical solutions, which we will state in a moment. Here, we would like to add that this trajectory statistical solution has a certain dissipation property of the vorticity as in \eqref{eq: a priori est vort}. Proving this will require the following lemma.

\begin{lem}\label{lem: measurability lp-norm vorticity}
The mappings $X_0^\infty \to \R, u \mapsto \|\omega(u)\|_{L^p(\R^2)}$ are Borel measurable for every $1 \leq p \leq \infty$. In particular, they are also Borel measurable when extended by $\infty$ to $\mathbb{E}$.
\end{lem}

\begin{proof}
Recall from \Cref{lem: measurability X^infty} that $X_0^\infty$ is Borel measurable in $\mathbb{E}$.\\ 
For every $1 \leq p < \infty$, $u \mapsto \|\omega(u)\|_{L^p(\R^2)}$ can be written as the pointwise limit $(\varepsilon \to 0)$ of the functions $u \mapsto \|\CalJ^\varepsilon \omega(u)\|_{L^p(B_{1/\varepsilon}(0))}$. For each $\varepsilon > 0$, these functions can be seen to be continuous on $X_0^\infty$ with respect to $\|\cdot\|_{\mathbb{E}}$. But then also $u \mapsto \|\omega(u)\|_{L^\infty(\R^2)}$ is Borel measurable as the pointwise limit $(p \to \infty)$ of the Borel measurable maps $u \mapsto \|\omega(u)\|_{L^p(\R^2)}$.\\
\end{proof}

In the following, $\Pi_0^\mathbb{E}: C([0,T];\mathbb{E}) \to \mathbb{E}$ is the time evaluation mapping at time $t = 0$. We are also going to use the constant $a$ as given in \eqref{eq: constant a}.

\begin{thm}\label{thm: ex traj stats sol infty}
Let $\mu_0$ be a Borel probability measure on $X_0^\infty$. The Borel probability measure $\rho^\infty$ on $C([0,T];\mathbb{E})$, given by
\[\rho^\infty(A) = S^\infty\mu_0(A) := \mu_0((S^\infty)^{-1}(A))\]
for all Borel measurable sets $A$ in $C([0,T];\mathbb{E})$, is the well-defined and unique $\CalU^\infty$-trajectory statistical solution satisfying $\Pi_0^\mathbb{E}\rho^\infty = \mu_0$. If $\int_{\mathbb{E}}a^{|m(u_0)|}\|u_{0,kin}\|_{L^2(\R^2)}\,d\mu_0(u_0) < \infty$, then
\begin{equation}\label{eq: energy inequ traj stats sols infty}
\int_{C([0,T];\mathbb{E})}\|u_{kin}(t)\|_{L^2(\R^2)}\,d\rho^\infty(u) \leq \int_{\mathbb{E}}a^{|m(u_0)|}\|u_{0,kin}\|_{L^2(\R^2)}\,d\mu_0(u_0)
\end{equation}
for every $0 \leq t \leq T$.\\
Moreover, if $\mu_0$ satisfies $\int_{\mathbb{E}} \|\omega(u_0)\|_{L^1(\R^2) \cap L^\infty(\R^2)}\,d\mu_0(u_0) < \infty$, then also
\begin{equation}\label{eq: energy est traj stats sol infty}
\int_{C([0,T];\mathbb{E})} \|\omega(u(t))\|_{L^1(\R^2) \cap L^\infty(\R^2)}\,d\rho^\infty(u) = \int_{\mathbb{E}} \|\omega(u_0)\|_{L^1(\R^2) \cap L^\infty(\R^2)}\,d\mu_0(u_0)
\end{equation}
for every $0 \leq t \leq T$.
\end{thm}

\begin{proof}
Due to the measurability of $S^\infty$, which was proved in \Cref{lem: measurability S infty nu}, $\rho^\infty$ is well-defined, clearly carried by $\CalU^\infty$, which is measurable due to \Cref{lem: meas U infty nu} and satisfies 
\[\Pi_0^\mathbb{E}\rho^\infty = (\Pi_0^\mathbb{E}\circ S^\infty)\mu_0 = \mu_0\]
since $\Pi_0^\mathbb{E}\circ S^\infty = \Id_{X_0^\infty}$.\\
We now prove \eqref{eq: energy est traj stats sol infty}. The energy inequality \eqref{eq: energy inequ traj stats sols infty} follows analogously based on \eqref{eq: energy ineq finite kin part}. Let $U: [0,T] \times C([0,T];\mathbb{E}) \to \mathbb{E}, (t,u) \mapsto u(t)$ be the evaluation mapping. As $U$ is continuous, it is also $\CalL([0,T])\otimes \CalB(C([0,T];\mathbb{E}))$-$\CalB(\mathbb{E})$ measurable. Using \Cref{lem: measurability S infty nu}, $U \circ (\Id_{[0,T]},S^\infty)$ is $\CalL([0,T])\otimes\CalB(\mathbb{E})$-$\CalB(\mathbb{E})$ measurable.\\
The fact that then also $(t,u_0) \mapsto \|\omega(S^\infty(u_0))(t)\|_{L^1(\R^2) \cap L^\infty(\R^2)}$ is $\CalL([0,T])\otimes\CalB(\mathbb{E})$ measurable follows from \Cref{lem: measurability lp-norm vorticity}.\\
Let $u_0 \in X_0^\infty$. Recall from \eqref{eq: a priori est vort} that for every $0 \leq t \leq T$, we have
\[\|\omega(S^\infty(u_0))(t)\|_{L^1(\R^2) \cap L^\infty(\R^2)} = \|\omega(u_0)\|_{L^1(\R^2) \cap L^\infty(\R^2)}.\]
Integrating this equality with respect to $\mu_0$ immediately yields \eqref{eq: energy est traj stats sol infty}.\\
To finish the proof, we show the claimed uniqueness. Let $\mu$ be any $\CalU^\infty$-trajectory statistical solution with initial distribution $\mu_0$ in the sense that $\Pi_0^\mathbb{E}\mu = \mu_0$.\\
Let $Q \in \CalB(C([0,T];\mathbb{E}))$. As $\mu$ is carried by the measurable set $\CalU^\infty$, we have
\begin{align}\label{eq: uniqueness traj stats sol}
\mu(Q) &= \mu(Q \cap \CalU^\infty) = \mu((\Pi_0^\mathbb{E})^{-1}(\Pi_0^\mathbb{E}(Q \cap \CalU^\infty)) \cap \CalU^\infty)\\
&= \mu((\Pi_0^\mathbb{E})^{-1}(\Pi_0^\mathbb{E}(Q \cap \CalU^\infty))) = \mu_0(\Pi_0^\mathbb{E}(Q \cap \CalU^\infty)),
\end{align}
where we used that $\Pi_0^\mathbb{E}(Q \cap \CalU^\infty) = (S^\infty)^{-1}(Q \cap \CalU^\infty)$ is measurable in $X^\infty = \mathbb{E}$ by \Cref{lem: measurability S infty nu} and \Cref{lem: meas U infty nu}. This particularly shows that $\mu$ is equal to $S^\infty\mu_0 = \rho^\infty$. 
\end{proof}

We likewise obtain the following theorem on existence of $\CalU^\nu_{NS}$-trajectory statistical solutions.

\begin{thm}\label{thm: ex traj stats sol NS}
Let $\mu_0$ be a Borel probability measure on $\mathbb{E}$. The Borel probability measure $\rho_{NS}^\nu$ on $C([0,T];\mathbb{E})$, given by
\[\rho_{NS}^\nu(A) = S_{NS}^\nu\mu_0(A) := \mu_0((S_{NS}^\nu)^{-1}(A))\]
for all Borel measurable sets $A$ in $C([0,T];\mathbb{E})$, is the well-defined and unique $\CalU^\nu_{NS}$-trajectory statistical solution satisfying $\Pi_0\rho_{NS}^\nu = \mu_0$.
\end{thm}

\begin{remark}
Similarly to \eqref{eq: energy inequ traj stats sols infty}, one can also obtain an energy inequality for trajectory statistical solutions of the Navier-Stokes equations as well as other properties, derived from the deterministic equations. We omit the discussion here as these properties are neither original nor do we need them in the following. Instead, we refer the reader for instance to Sections $6$ and $7$ in \cite{Kelliher2009}.
\end{remark}

To close this section, we briefly argue in what way the $\CalU^\infty$-trajectory statistical solution in \Cref{thm: ex traj stats sol infty} can also be obtained in the inviscid limit of the $\CalU^\nu_{NS}$-trajectory statistical solution in \Cref{thm: ex traj stats sol NS} as $(\nu \to 0)$.\\
The following theorem due to Chemin (see \cite{Chemin1996}) in the deterministic case will make things quite convenient.

\begin{thm}\label{thm: inviscid lim deterministic}
Let $u_0 \in X_0^\infty$. Then
\[\lim_{\nu \to 0}S_{NS}^{\nu}(u_0) = S^\infty(u_0) \text{ in } C([0,T];\mathbb{E}).\]
\end{thm}

Then, by a simple application of the dominated convergence theorem, we obtain the following result from \Cref{thm: inviscid lim deterministic}.

\begin{thm}\label{thm: traj inviscid limit infty}
Let $\mu_0$ be a Borel probablity measure on $X_0^\infty$. Then $(S^\nu_{NS}\mu_0)_{\nu > 0}$ converges to $S^\infty\mu_0$ as $(\nu \to 0)$ in the sense that for every real-valued, bounded continuous function $\Phi$ on $C([0,T];\mathbb{E})$,
\begin{equation}\label{eq: inviscid limit infty}
\int_{C([0,T];\mathbb{E})} \Phi(u)\,dS_{NS}^\nu\mu_0(u) \to \int_{C([0,T];\mathbb{E})} \Phi(u)\,dS^\infty\mu_0(u)\,(\nu \to 0).
\end{equation}
\end{thm}

\subsection{Vorticity in $L^p(\R^2)$, $1 \leq p < \infty$, and in $\CalM^+(\R^2)$}
In this subsection, we construct trajectory statistical solutions for the class of weak solutions of the Euler equations with vorticity in $L^p(\R^2)$, $1 \leq p < \infty$, and in the space $\CalM^+(\R^2)$ of finite, non-negative Borel measures, considered by Delort.\\
We fix $1 < p < \infty$ and define the spaces
\begin{itemize}
\item $X^p = H_{loc}$, $X^{VS} = X^1 = H^{-L}_{loc}(\R^2;\R^2)$;
\item $X_0^p = \lbrace u_0 \in \mathbb{E}: \omega(u_0) \in L^1(\R^2)\cap L^p(\R^2)\rbrace$;
\item $X_0^{VS} = \lbrace u_0 \in \mathbb{E}: \omega(u_0) \in \CalM_c^+(\R^2)\rbrace$;
\item $X_0^1 = \lbrace u_0 \in \mathbb{E}: \omega(u_0) \in L^1_c(\R^2)\rbrace$;
\item $\CalU^p$, the set of weak solutions of the Euler equations $u \in C([0,T];X^p) \cap L^\infty(0,T;$ $W^{1,p}_{loc}(\R^2;\R^2) \cap \mathbb{E})$ whose vorticity $\omega(u) \in C([0,T];L^p(\R^2))$ is a renormalized solution of the vorticity formulation of the Euler equations having initial data in $X_0^p$ and satisfying \eqref{eq: energy ineq finite kin part}, \eqref{eq: a priori est time derivative} and \eqref{eq: a priori W^1,p};
\item $\CalU^{VS}$, the set of weak solutions of the Euler equations $u \in C([0,T];X^{VS}) \cap L^\infty(0,T;\mathbb{E})$ with vorticity $\omega(u) \in C_w([0,T];\CalM(\R^2)) \cap L^\infty(0,T;\CalM^+(\R^2))$ and initial data in $X_0^{VS}$ satisfying
\eqref{eq: energy ineq finite kin part}, \eqref{eq: a priori est time derivative} and \eqref{eq: a priori est vort L1,VS};
\item $\CalU^1$, the set of weak solutions of the Euler equations $u \in C([0,T];X^1) \cap L^\infty(0,T;\mathbb{E})$ with vorticity $\omega(u) \in C_w([0,T];\CalM(\R^2)) \cap L^\infty(0,T;L^1(\R^2))$ and initial data in $X_0^1$ satisfying \eqref{eq: energy ineq finite kin part}, \eqref{eq: a priori est time derivative}, \eqref{eq: a priori est vort L1,VS} and \eqref{eq: local conservation vorticity}.
\end{itemize}
As in the previous subsection, we will typically write out $H_{loc}$ or $H^{-L}_{loc}(\R^2;\R^2)$ in the following instead of $X^p, X^{VS}$ or $X^1$ in order to make it clearer which topology we consider.\\
We will treat these cases simultaneously whenever it is possible and seems reasonable. Should certain aspects vary too much though, then we will state and prove the results separate from one another. For the occasions where we consider all cases at once or formulate statements which hold in each case, we may write $X, X_0$ and $\CalU$ as placeholders.\\
Unlike the weak solutions considered in the previous subsection, the weak solutions considered here may no longer be unique and we will make use of the abstract existence result \Cref{thm: ex traj stats sol v2}. Even though it would not be necessary, we will in fact demonstrate the nice arguments used in the proof of \Cref{thm: ex traj stats sol v2} in the particularly simple situation when the vorticity is in $L^p(\R^2)$.\\
To begin with, let us prove a few measurability properties of these spaces.

\begin{lem}\label{lem: borel algebras on E}
The Borel-$\sigma$-algebras on $\mathbb{E}$ generated by the topology induced by $\|\cdot\|_{\mathbb{E}}$ and the (subspace-) topologies of $H_{loc}$ and $H^{-L}_{loc}(\R^2;\R^2)$ coincide. Moreover, the set $\mathbb{E}$ is Borel measurable itself in $H_{loc}$ and $H^{-L}_{loc}(\R^2;\R^2)$.
\end{lem}

\begin{proof}
We will only explicitly prove the case of $H_{loc}$. Showing measurability of $\mathbb{E}$ in $H^{-L}_{loc}(\R^2;\R^2)$ and equality between the Borel-$\sigma$-algebras as stated in the lemma can be done in the exact same way by simply replacing $H_{loc}$ with $H_{loc}^{-L}(\R^2;\R^2)$ in the following.\\
Due to the continuous embedding $\mathbb{E} \hookrightarrow H_{loc}$, relatively open sets in $\mathbb{E}$ with respect to the topology of $H_{loc}$ are also $\|\cdot\|_{\mathbb{E}}$ open. This implies one inclusion between the two Borel-$\sigma$-algebras.\\
Next, we note that since $\mathbb{E}$ is a separable Banach space, the Borel-$\sigma$-algebra on $\mathbb{E}$ generated by the norm topology and the weak topology coincide. The latter is also generated by the weakly compact subsets, as $\mathbb{E} = \bigcup_{k=1}^\infty \overline{B}^\mathbb{E}_k(0)$ with the weak topology is $\sigma$-compact. Therefore, to prove the other inclusion stated in the lemma, we show that weakly compact subsets of $\mathbb{E}$ are closed in $H_{loc}$. Both points just mentioned will then also yield the measurability of $\mathbb{E}$ in $H_{loc}$.\\
Let $K$ be such a weakly compact subset in $\mathbb{E}$ and consider a sequence $(u^n)_{n\in\N}$ in $K$ which converges to some $u \in H_{loc}$ with respect to the metric of $H_{loc}$. By the Eberlein-\v{S}mulian theorem, $K$ is also weakly sequentially compact and, after passing to a subsequence of $(u^n)_{n\in\N}$, there exists some $v \in K$ such that $u^n \rightharpoonup v\,(n\to\infty)$ in $\mathbb{E}$. Weak convergence of $(u^n)_{n\in\N} = (m(u^n)\Sigma + u^n_{kin})_{n\in\N}$ in $\mathbb{E}$ is equivalent to $(m(u^n))_{n\in\N}$ converging in $\R$ and $(u^n_{kin})_{n\in\N}$ converging weakly in $H$. Consequently, $(u^n)_{n\in\N}$ converges weakly in $L^2(B)$ to $v$ for every bounded measurable subset $B \subset \R^2$.\\
By testing with some $\varphi \in C_c^\infty(\R^2)$, the two types of convergence yield $u = v$, which particularly implies $u \in K$ and thereby the closedness of $K$. 
\end{proof}

Recall that in \Cref{thm: ex traj stats sol v2}, a fundamental assumption was the measurability of $X_0$ in $X$, which we are going to check in the following lemmata.

\begin{lem}\label{lem: measurability X^p X^VS}
The spaces $X_0^p$, $1 < p < \infty$, and $X_0^{VS}$ are Borel measurable in $\mathbb{E}$. In particular, they are also Borel measurable subsets of $H_{loc}$ and $H^{-L}_{loc}(\R^2;\R^2)$. 
\end{lem}

\begin{proof}
Both cases may be proved similarly to \Cref{lem: measurability X^infty} by writing 
\[X_0^p = \bigcup_{k=1}^\infty \lbrace u \in X_0^p : \|\omega(u)\|_{L^1(\R^2)}, \|\omega(u)\|_{L^p} \leq k\rbrace\]
and 
\[X_0^{VS} = \bigcup_{k = 1}^\infty \lbrace u \in X_0^{VS} : \supp\omega(u) \subset \overline{B}_k(0)\rbrace.\]
The remaining part of measurability of $X_0^p$ and $X_0^{VS}$ in $H_{loc}$ and $H^{-L}_{loc}(\R^2;\R^2)$ now follows from \Cref{lem: borel algebras on E}.
\end{proof}

Showing the measurability of $X_0^1$ in $\mathbb{E}$ or $H^{-L}_{loc}(\R^2;\R^2)$ is a bit more difficult and has to be treated differently, compared to the previous cases, due to the lack of reflexivity of $L^1(\R^2)$. We first prove the following lemma.

\begin{lem}\label{lem: measurability L1 in M}
The spaces $L^1(\R^2)$ and $L_c^1(\R^2)$ are Borel measurable in $\CalM(\R^2)$ endowed with the weak-* topology. 
\end{lem}

\begin{proof}
We first note that from the definition of the total variation norm via the duality $\CalM(\R^2) \cong C_0(\R^2)'$, it follows that the total variation norm is weakly lower semicontinuous as a mapping from $\CalM(\R^2)$, endowed with the weak-* topology, into the real numbers. Consequently, open balls with respect to the total variation norm are measurable with respect to the Borel-$\sigma$-algebra generated by the weak-* topology.\\
Let $(f^n)_{n\in\N}$ be a dense subset of $L^1(\R^2)$. To conclude the proof of measurability of $L^1(\R^2)$, it now suffices to show
\begin{equation}\label{eq: measurability L1 in M}
L^1(\R^2) = \bigcap_{k\in\N}\bigcup_{n\in\N} \bigg\lbrace \mu \in \CalM(\R^2): \int_{\R^2}|f^n-\mu| < \frac{1}{k}\bigg\rbrace.
\end{equation}
Indeed, a measure $\mu \in \CalM(\R^2)$ is an element of the set on the right-hand side if and only if there exists a strictly increasing sequence $(n_m)_{m\in\N} \subset \N$ such that $f^{n_m} \to \mu$ $(m \to \infty)$ with respect to the total variation norm. In particular, $(f^{n_m})_{m\in\N}$ is a Cauchy sequence in $\CalM(\R^2)$ with respect to the total variation norm. As the total variation norm and the $L^1(\R^2)$ norm coincide on $L^1(\R^2) \subset \CalM(\R^2)$, $(f^{n_m})_{m\in\N}$ is also a Cauchy-sequence in $L^1(\R^2)$. Due to the completeness of $L^1(\R^2)$, necessarily $\mu \in L^1(\R^2)$.\\
The other inclusion in \eqref{eq: measurability L1 in M} follows from the density of $(f^n)_{n\in\N}$.\\
As for $L_c^1(\R^2)$, we may write 
\[L_c^1(\R^2) = L^1(\R^2) \cap \bigcup_{k=1}^\infty \lbrace \omega \in \CalM(\R^2): \supp\omega \subset \overline{B}_k(0)\rbrace,\]
where each set in the countable union is closed in $\CalM(\R^2)$ with respect to the weak-* topology. 
\end{proof}

\begin{lem}\label{lem: measurability X^1}
The space $X_0^{1}$ is Borel measurable in $\mathbb{E}$. In particular, it is also a Borel measurable subset of $H^{-L}_{loc}(\R^2;\R^2)$. 
\end{lem}

\begin{proof}
In this proof, measurability in $\CalM(\R^2)$ is always meant with respect to the Borel-$\sigma$-algebra generated by the weak-* topology as in the previous lemma.\\
We introduce the mapping $\Psi: \lbrace u \in \mathbb{E}: \omega(u) \in \CalM(\R^2)\rbrace \to \CalM(\R^2), u \mapsto \omega(u)$.\\
Consider the closed and hence measurable subsets $\mathbb{E}_k = \lbrace u \in \mathbb{E} : \omega(u) \in \CalM(\R^2), \int_{\R^2} |\omega(u)| \leq k \rbrace$ of $\mathbb{E}$ for every $k \in \N$ and denote by $\Psi_k$ the restrictions of $\Psi$ to $\mathbb{E}_k$ for all $k \in \N$. We argue that every mapping $\Psi_k$ is continuous: Let $(u^n)_{n\in\N}$ be a sequence in $\mathbb{E}_{k}$ converging to some $u \in \mathbb{E}$. Then $(\omega(u^n))_{n\in\N}$ is a sequence in the closed total variation ball of radius $k$, centred at $0$. By the Banach-Alaoglu theorem there exists a subsequence, again denoted by $(\omega(u^n))_{n\in\N}$, converging weakly-* to some $\omega \in \CalM(\R^2)$. From convergence in $\mathbb{E}$, it follows that necessarily $\omega = \omega(u)$ and the weak-* convergence already holds for the sequence itself, i.e.\ $\lim_{n\to\infty} \Psi_k(u^n) = \Psi_k(u)$ with respect to the weak-* topology.\\
Finally, $X_0^1 = \bigcup_{k\in\N} X_0^1 \cap \mathbb{E}_k = \bigcup_{k\in\N}\Psi_k^{-1}(L^1_c(\R^2))$ is measurable by \Cref{lem: measurability L1 in M}.\\
\end{proof}

To meet the requirements of \Cref{thm: ex traj stats sol}, we will need to construct families $\CalK'(X_0)$ of compact subsets of the space of initial data such that solutions having initial data in one of those compact sets are compact in the whole set of solutions. We remark here that this required closedness in particular includes that the limits need not only be weak solutions of the Euler equations, but also satisfy the a priori estimates that we impose in the definition of $\CalU^p, \CalU^{VS}$ and $\CalU^1$. Estimating the left-hand side may usually be done by an argument of weak or weak-* lower semicontinuity. Estimating the right-hand side, which involves the initial data, is a different matter. In the construction of weak solutions, convergence of the right-hand side involving the initial data is usually obvious from the approximation scheme, e.g. smoothing of the initial data. For general sequences of solutions as considered here, it is not clear. In the case of the three-dimensional Navier-Stokes equations considered for instance in \cite{Foias2013}, the same problem arises as it is not clear whether weak limits of Leray-Hopf solutions satisfy the energy inequality.\\
Therefore, the sets in $\CalK'(X_0)$ particularly need to be compact in a compatible way with the a priori estimates demanded above.\\
Interestingly enough, no infinite dimensional Banach space may be exhausted by compact subsets, which can be seen by a Baire category argument. Yet in a measure theoretic way this is possible, if the space is Polish, in the sense that every Borel measure is inner regular (see Theorem 3.2 in \cite{Parthasarathy1967}). This fact will be heavily exploited later on.\\
For the precise definition of those families of compact sets, we introduce the operator $\curl^p: X_0^p \to L^1(\R^2) \cap L^p(\R^2), u_0 \mapsto \omega(u_0)$, as well as the operator $\curl^1: X_0^1 \to L^1(\R^2), u_0 \mapsto \omega(u_0)$.\\
Using the smoothing operator as in \Cref{lem: measurability lp-norm vorticity}, it is not hard to see that $\curl^p$ and $\curl^1$ are measurable as pointwise limits of continuous functions and may then also be seen as measurable maps on $\mathbb{E}$, extended by $\infty$, due to \Cref{lem: measurability X^p X^VS} and \Cref{lem: measurability X^1}.
Now we define
\begin{itemize}
\item $\CalK'(X_0^p)$ to be the family consisting of all sets of the form $S \cap (\curl^p)^{-1}(K)$, where $S \subset X_0^p$ is compact in $\mathbb{E}$ and $K$ is a compact set in $L^1(\R^2) \cap L^p(\R^2)$;
\item $\CalK'(X_0^{VS})$ to be the family consisting of all subsets of $X_0^{VS}$ which are compact in $\mathbb{E}$;
\item $\CalK'(X_0^1)$ to be the family consisting of all sets of the form $S \cap (\curl^1)^{-1}(K)$, where $S \subset X_0^1$ is compact in $\mathbb{E}$ and $K$ is a compact set in $L^1(\R^2)$.
\end{itemize}
For a compact set $K$ in $L^1(\R^2) \cap L^p(\R^2)$, $(\curl^p)^{-1}(K)$ is a subset of $X_0^p$, (relatively) closed in $\mathbb{E}$, which implies that with $S$ as above, $S \cap (\curl^p)^{-1}(K)$ is a subset of $X_0^p$ which is compact in $\mathbb{E}$. Due to the continuous embedding $\mathbb{E} \hookrightarrow H_{loc}$, the sets in $\CalK'(X_0^p)$ are indeed compact in $H_{loc}$.\\
Likewise, one may argue that $\CalK'(X^1_0)$ consists of sets in $X_0^1$ which are compact in $\mathbb{E} \hookrightarrow H^{-L}_{loc}(\R^2;\R^2)$.\\
Compactness of the sets in $\CalK'(X_0^{VS})$ in $\mathbb{E}\hookrightarrow H_{loc}^{-L}(\R^2;\R^2)$ follows directly from the definition.

We now show that these families satisfy \textit{(H2')} in \Cref{thm: ex traj stats sol v2}.

\begin{lem}\label{lem: inner reg vort}
Let $\mu_0$ be a Borel probability measure on $X_0$. Then $\mu_0$ is inner regular with respect to the family $\CalK'(X_0)$. 
\end{lem}

\begin{proof}
We begin with the easiest case of $\mu_0$ being a Borel probability measure on $X_0^{VS}$. By \Cref{lem: borel algebras on E} and \Cref{lem: measurability X^p X^VS}, $\mu_0$ can be seen as a Borel probability measure on $\mathbb{E}$. Since $\mathbb{E}$ is a Polish space, $\mu_0$ is inner regular with respect to the $\|\cdot\|_{\mathbb{E}}$ compact sets in $X_0^{VS}$, which form by definition $\CalK'(X_0^{VS})$.\\
Now we consider the case of $\mu_0$ being a Borel probability measure on $X_0^p$. Let $A$ be a Borel subset of $X_0^p$. Due to \Cref{lem: borel algebras on E} and \Cref{lem: measurability X^p X^VS}, we may also view $\mu_0$ as a measure on $\mathbb{E}$, which is a Polish space. In this sense, $\mu_0$ is already inner regular and we obtain
\[\mu_0(A) = \sup_{\substack{S \subset A\\S \text{ compact in }\mathbb{E}}}\mu_0(S) = \sup_{\substack{S \subset A\\S \text{ compact in }\mathbb{E}}}\mu_0((\curl^p)^{-1}(\curl^p(S)))\]
We would now like to argue that for any compact subset $S$ in $\mathbb{E}$, $\curl^p(S)$ is Borel measurable in $L^1(\R^2) \cap L^p(\R^2)$. First of all, let $T: L^1(\R^2) \cap L^p(\R^2) \to L^1_{loc}(\R^2;\R^2)$ be the Biot-Savart operator (see \Cref{lem: Biot-Savart}). Note that $T$ is continuous even considered as an operator on $L^1(\R^2)$ as it is a convolution operator with a kernel that may be written as a sum in $L^1(\R^2) + L^\infty(\R^2)$. By \Cref{lem: Biot-Savart}, we have
\begin{equation}\label{eq: curl T}
\curl^p(S) = T^{-1}(S).
\end{equation}
Since convergence in $\mathbb{E}$ implies convergence in $H_{loc}$, which again implies convergence in $L^1_{loc}(\R^2;\R^2)$, $S$ is not only compact with respect to $\|\cdot\|_{\mathbb{E}}$ but also with respect to the topology of $L^1_{loc}(\R^2;\R^2)$. Due to \eqref{eq: curl T}, $\curl^p(S)$ is closed and hence Borel measurable in $L^1(\R^2) \cap L^p(\R^2)$.\\
We may then introduce the pushforward measure $\mu_0^p := \curl^p\mu_0$ on $L^1(\R^2) \cap L^p(\R^2)$. As a Borel probability measure on a Polish space, $\mu_0^p$ is also inner regular.\\
Continuing where we left of with our computation of $\mu_0(A)$, we now have 
\begin{align}
\mu_0(A) &= \sup_{\substack{S \subset A\\S \text{ compact in }\mathbb{E}}}\mu_0^p(\curl^p(S))\\
&= \sup_{\substack{S \subset A\\S \text{ compact in }\mathbb{E}}} \sup_{\substack{K \subset \curl^p(S)\\K \text{ compact in }L^1(\R^2)\cap L^p(\R^2)}} \mu_0^p(K)\\
&= \sup_{\substack{S \subset A\\S \text{ compact in }\mathbb{E}}} \sup_{\substack{(\curl^p)^{-1}(K) \subset S\\K \text{ compact in }L^1(\R^2)\cap L^p(\R^2)}} \mu_0((\curl^p)^{-1}(K)\cap S)\\
&\leq \sup_{\substack{S \cap (\curl^p)^{-1}(K) \subset A\\S \subset X_0^p \text{ compact in }\mathbb{E}\\K \text{ compact in }L^1(\R^2)\cap L^p(\R^2)}}\mu_0((\curl^p)^{-1}(K)\cap S)\\
&\leq \mu_0(A),
\end{align}
as desired. The remaining case of $\mu_0$ being a Borel probability measure on $X_0^1$ may be proved analogously to the previous one due to the similarity between the definition of $\CalK'(X_0^p)$ and $\CalK'(X_0^1)$ and due to the fact that the Biot-Savart operator $T$ above is still continuous when changing its domain to $L^1(\R^2)$.
\end{proof}

Next, we prove that our families of compact sets also satisfy \textit{(H3')} in \Cref{thm: ex traj stats sol v2}. For this, let $\Pi^{X^p}_0: C([0,T];H_{loc}) \to H_{loc}$ and $\Pi^{X^1}_0,\Pi^{X^{VS}}_0: C([0,T];H^{-L}_{loc}(\R^2;\R^2)) \to H^{-L}_{loc}(\R^2;\R^2)$ be the time evaluation maps between the given spaces at time $t = 0$. When we discuss all cases simultaneously, we will simply write $\Pi_0^X$.

\begin{lem}\label{lem: comp Lp solutions}
Let $S$ be a subset of $X^p_0$ which is compact in $\mathbb{E}$ and let $K$ be a compact subset of $L^1(\R^2) \cap L^p(\R^2)$ such that $S \cap (\curl^p)^{-1}(K)$ is an element of $\CalK'(X_0^p)$. Then the set $\kappa^p := (\Pi^{X^p}_0)^{-1}(S \cap (\curl^p)^{-1}(K)) \cap \CalU^p$ is compact in $C([0,T];H_{loc})$.
\end{lem} 

\begin{proof} 
Consider an arbitrary sequence $(u^n)_{n\in\N} \subset \kappa^p$.
From \eqref{eq: a priori W^1,p}, we obtain uniform boundedness of $(u^n)_{n\in\N}$ in $L^\infty(0,T;W^{1,p}_{loc}(\R^2;\R^2))$. Moreover $(\partial_t u^n)_{n\in\N}$ is uniformly bounded in $L^\infty(0,T;H^{-L}(\R^2;\R^2))$ for some $L > 1$ by \eqref{eq: a priori est time derivative}. Due to the compact embeddings
\[W^{1,p}_{loc}(\R^2;\R^2) \cap H_{loc} \hookrightarrow H_{loc} \hookrightarrow H_{loc}^{-L}(\R^2;\R^2),\]
the Aubin-Lions lemma and a diagonal sequence argument imply convergence of a subsequence of $(u^n)_{n\in\N}$ in $C([0,T];H_{loc})$ to some $u$, so it only remains to show that $u \in \kappa^p$.\\
As each field in the sequence $(u^n)_{n\in\N}$ satisfies the weak velocity formulation of the Euler equations given in \eqref{def: Euler weak vel}, the precompactness in $C([0,T];H_{loc})$ allows us to pass to the limit in all terms, in particular the quadratic one, so that $u$ is also a weak solution of the Euler equations.\\
Moreover, from compactness of $S$ in $\mathbb{E}$ and from compactness of $K$ in $L^1(\R^2) \cap L^p(\R^2)$, we may assume
\begin{align}
u^n(0) &\to u(0)\,(n\to\infty) \text{ in } \mathbb{E}\\
\omega(u^n)(0) &\to \omega(u)(0) \,(n\to\infty) \text{ in } L^1(\R^2) \cap L^p(\R^2).
\end{align}
Therefore, $u(0) \in K$ and properties \eqref{eq: energy ineq finite kin part}, \eqref{eq: a priori est time derivative} and \eqref{eq: a priori W^1,p}, demanded in the definition of $\CalU^p$, can be seen to be satisfied by $u$ from combining this strong convergence of the initial data along with an argument of weak or weak-* lower semicontinuity involving $(u^n)_{n\in\N}$ and $(\omega(u^n))_{n\in\N}$ in the respective space. Also, the strong convergence of the initial vorticities shows that after passing yet to another subsequence that $\omega(u)$ is a renormalized solution of the vorticity formulation of the Euler equations (see \Cref{lem: stab renormalized sols}). In particular, $u \in \kappa^p$.
\end{proof}

\begin{lem}\label{lem: comp VS solutions}
Let $K$ be a compact set in the family $\CalK'(X_0^{VS})$. Then the set $\kappa^{VS} := (\Pi^{X^{VS}}_0)^{-1}(K) \cap \CalU^{VS}$ is compact in $C([0,T];H^{-L}_{loc}(\R^2;\R^2)).$
\end{lem}

\begin{proof}
Let $(u^n)_{n\in\N} \subset \kappa^{VS}$ be a sequence. Since $\omega(u^n)(0) \in \CalM^+_c(\R^2)$ for every $n \in \N$, by our comments preceding \Cref{thm: solutions of euler}, $u_0^n = m(u_0^n)\Sigma + u_{0,kin}^n$ with $m(u_0^n) = \int_{\R^2} \omega(u_0^n)$. Therefore, by compactness of $K$ in $\mathbb{E}$, this implies that (a subsequence of) the total mass of initial vorticity converges. Hence, $(\partial_t u^n)_{n\in\N}$ is uniformly bounded in $L^\infty(0,T;H^{-L}(\R^2;\R^2))$ due to \eqref{eq: a priori est time derivative}. The bound on the local kinetic energy \eqref{eq: comp bdd L2_loc weak sols} and the definition of $K$ yield that $(u^n)_{n\in\N}$ is uniformly bounded in $L^\infty(0,T;H_{loc})$, where $H_{loc}$ compactly embeds into $H^{-L}_{loc}(\R^2;\R^2)$. Consequently, $(u^n)_{n\in\N}$ is equicontinuous in $C([0,T];H^{-L}_{loc}(\R^2;\R^2))$ with values in a relatively compact subset of $H^{-L}_{loc}(\R^2;\R^2)$. Then the Arzel\`{a}-Ascoli theorem implies that after passing to a subsequence, $(u^n)_{n\in\N}$ converges in $C([0,T];H_{loc}^{-L}(\R^2;\R^2))$ to some $u$, so it only remains to show $u \in \kappa^{VS}$. Due to \eqref{eq: energy ineq finite kin part} and the compactness of $K$ in $\mathbb{E}$, $(u^n)_{n\in\N}$ is uniformly bounded in $L^\infty(0,T;\mathbb{E})$, so that $u$ may also be seen as the weak-* limit in that space.\\
We are now going to apply \Cref{thm: M L^1 limit result}. Due to \eqref{eq: a priori est vort L1,VS}, $(\omega(u^n))_{n\in\N}$ is uniformly bounded in $L^\infty(0,T;\CalM^+(\R^2))$. Then $\omega(u)$ can be seen as the weak-* limit of $(\omega(u^n))_{n\in\N} = (|\omega(u^n)|)_{n\in\N}$ in $L^\infty(0,T;\CalM^+(\R^2))$. Since $u \in L^\infty(0,T;\mathbb{E})$, $\omega(u)(t)$ is continuous for almost every $0 \leq t \leq T$ (see \Cref{rem: E and H^-1} and the comments prior to \Cref{thm: M L^1 limit result}). Therefore, \Cref{thm: M L^1 limit result} yields that $u$ is indeed a weak solution of the Euler equations.\\
The a priori estimates \eqref{eq: energy ineq finite kin part}, \eqref{eq: a priori est time derivative} and \eqref{eq: a priori est vort L1,VS}, demanded in the definition of $\CalU^{VS}$, are also satisfied by the weak limit $u$. Indeed, one may estimate the left-hand side by classic weak-* compactness arguments along with weak lower semicontinuity in an appropriate sense. The right-hand side in each estimate involving the initial data converges due to the compactness of $K$ with respect to $\|\cdot\|_{\mathbb{E}}$.\\
As for the continuity of $\omega(u)$ demanded in the definition of $\CalU^{VS}$, from boundedness of $(\partial_t u^n)_{n\in\N}$ in $L^\infty(0,T;H^{-L}(\R^2;\R^2))$ we obtain boundedness of $(\partial_t \omega(u^n))_{n\in\N}$ in the space $L^\infty(0,T;H^{-L-1}(\R^2;\R^2))$. Then, as remarked in \cite{Schochet1995}[Lemma 3.2], the weak-* convergence of $(\omega(u^n))_{n\in\N}$ to $\omega(u)$ in $L^\infty(0,T;\CalM(\R^2))$ allows one to derive that $(\omega(u^n))_{n\in\N}$ converges in $C_w([0,T];\CalM(\R^2))$ to $\omega(u)$. 
\end{proof}

\begin{lem}\label{lem: comp L1 solutions}
Let $S$ be a subset of $X_0^1$, which is compact in $\mathbb{E}$ and let $K$ be a compact subset of $L^1(\R^2)$ such that $S \cap (\curl^1)^{-1}(K)$ is an element of $\CalK'(X_0^1)$. 
Then the set $\kappa^1 := (\Pi^{X^1}_0)^{-1}(S \cap (\curl^1)^{-1}(K)) \cap \CalU^1$ is compact in $C([0,T];H^{-L}_{loc}(\R^2;\R^2)).$
\end{lem}

\begin{proof}
Let $(u^n)_{n\in\N} \subset \kappa^1$ be a sequence. The precompactness of $(u^n)_{n\in\N}$ in the space $C([0,T];H^{-L}_{loc}(\R^2;\R^2))$ follows again from the Aubin-Lions lemma as in \Cref{lem: comp VS solutions}. We denote the limit of a subsequence by $u$ and keep on writing $(u^n)_{n\in\N}$ for this and for further subsequences. Due to the compactness of $K$ in $L^1(\R^2)$, we may also suppose that $\omega(u^n)(0) \to \omega(u)(0)$ in $L^1(\R^2)$ and we can also see $u$ as the weak-* limit of $(u^n)_{n\in\N}$ in $L^\infty(0,T;\mathbb{E})$ as this sequence is bounded in that space by \eqref{eq: energy ineq finite kin part}.\\
We now argue that $u$ satisfies \eqref{eq: a priori est vort L1,VS}. As at the end of the proof of \Cref{lem: comp VS solutions}, we not only obtain weak-* convergence of $(\omega(u^n))_{n\in\N}$ in $L^\infty(0,T;\CalM(\R^2))$ to $\omega(u)$, but in fact convergence also holds in $C_w([0,T];\CalM(\R^2))$. We particularly obtain $\omega(u^n)(t) \overset{\ast}{\rightharpoonup} \omega(u)(t)\,(n\to\infty)$ in $\CalM(\R^2)$ for every $0 \leq t \leq T$. We combine this with the fact that due to \eqref{eq: local conservation vorticity} and the Dunford-Pettis theorem, by which uniform integrability and weak (sequential) precompactness in $L^1_{loc}(\R^2)$ are equivalent (see Chapter 4 in \cite{Bogachev2007} and in particular Theorem 4.7.18), $(\omega(u^n)(t))_{n\in\N}$ is weakly precompact in $L^1_{loc}(\R^2)$ for every $0 \leq t \leq T$. We then necessarily obtain $\omega(u^n)(t) \rightharpoonup \omega(u)(t)$ in $L^1_{loc}(\R^2)$ for every $0 \leq t \leq T$. Weak lower semicontinuity of the norm then yields for every $r > 0$
\begin{align}
\|\omega(u)(t)\|_{L^1(B_r(0))} &\leq \liminf_{n\to\infty}\|\omega(u^n)(t)\|_{L^1(B_r(0))} \leq \liminf_{n\to\infty}\|\omega(u^n)(t)\|_{L^1(\R^2)}\\
&\leq \liminf_{n\to\infty}\|\omega(u^n)(0)\|_{L^1(\R^2)} = \|\omega(u)(0)\|_{L^1(\R^2)}
\end{align}
for every $0 \leq t \leq T$. This means that $u$ also
satisfies \eqref{eq: a priori est vort L1,VS}, which we demanded in the definition of $\CalU^1$.\\
Likewise, one may argue that $|\omega(u^n)|$ converges weakly-* in $L^\infty(0,T;\CalM(\R^2))$ to an element of $L^\infty(0,T;L^1(\R^2))$ so that the continuity assumption in \Cref{thm: M L^1 limit result} is satisfied. Therefore, $u$ is indeed a weak solution of the Euler equations.\\
The properties \eqref{eq: energy ineq finite kin part}, \eqref{eq: a priori est time derivative} and \eqref{eq: a priori est vort} demanded in the definition of $\CalU^1$ follow as in \Cref{lem: comp VS solutions}. Property \eqref{eq: local conservation vorticity} follows from strong convergence $\lim_{n\to\infty} \omega(u^n)(0) = \omega(u)(0)$ in $L^1(\R^2)$ and from weak-* lower semicontinuity due to weak-* convergence of $(\omega(u^n))_{n\in\N}$ in $L^\infty(0,T;\CalM^+(\R^2))$.
\end{proof}

We are now able to state the existence result of trajectory statistical solutions in the classes considered in this subsection. As mentioned in the introduction to this subsection, for the case of $X^p = H_{loc}$ being the phase space, we will provide the arguments used in the proof of \Cref{thm: ex traj stats sol}, applied to our specific situation for illustrative purposes. In the other two cases we will directly refer to the statement of \Cref{thm: ex traj stats sol v2}.\\
In the following, for $u \in X$, we define
\begin{equation}\label{eq: norm vorticity}
\|\omega(u)\|_X := \begin{cases}
\|\omega(u)\|_{L^1(\R^2) \cap L^p(\R^2)} &: X = X^p\\
\int_{\R^2}\omega(u_0) &: X = X^{VS}\\
\|\omega(u)\|_{L^1(\R^2)} &: X = X^1
\end{cases},
\end{equation}
with the right-hand side being $\infty$ if $\omega(u)$ does not lie in the appropriate space. For the next theorem, recall the constant $a$ as given in \eqref{eq: constant a}.

\begin{thm}\label{thm: ex traj stats sol Euler}
Let $\mu_0$ be an arbitrary Borel probability measure on $X_0$. Then there exists a $\CalU$-trajectory statistical solution $\rho$ satisfying $\Pi_0^X\rho  = \mu_0$.\\
If $\int_{\mathbb{E}}a^{|m(u_0)|}\|u_{0,kin}\|_{L^2(\R^2)}\,d\mu_0(u_0) < \infty$, then
\begin{equation}\label{eq: energy inequ traj stats sols}
\int_{C([0,T];\mathbb{E})}\|u_{kin}(t)\|_{L^2(\R^2)}\,d\rho(u) \leq \int_{\mathbb{E}}a^{|m(u_0)|}\|u_{0,kin}\|_{L^2(\R^2)}\,d\mu_0(u_0)
\end{equation}
for almost every $0 \leq t \leq T$.\\
Moreover, if $\mu_0$ satisfies $\int_{X} \|\omega(u_0)\|_{X}\,d\mu_0(u_0) < \infty$, then
\begin{equation}\label{eq: energy inequ vort traj stats sol}
\int_{C([0,T];X)} \|\omega(u(t))\|_{X}\,d\rho(u) \leq \int_{X} \|\omega(u_0)\|_{X}\,d\mu_0(u_0),
\end{equation}
for every $0 \leq t \leq T$ with equality in case $X = X^p$.
\end{thm}

\begin{proof}
We consider the case of $\mu_0$ being a Borel probability measure on $X_0^p$. Assume that $\mu_0$ is concentrated on some compact set $\kappa_0^p := S \cap (\curl^p)^{-1}(K)$ in the family $\CalK'(X^p_0)$, where $S$ is a subset of $X^p_0$ which is compact with respect to $\|\cdot\|_{\mathbb{E}}$ and $K$ is a compact set in $L^1(\R^2) \cap L^p(\R^2)$.\\ 
Then $C(\kappa_0^p)'$ is the space of finite (signed) Borel measures on $\kappa_0^p$. As $C(\kappa_0^p)$ is separable due to the compactness of $\kappa_0^p$, the unit ball $\mathbb{B}$ in $C(\kappa_0^p)'$ is weak-* sequentially compact by the Banach-Alaoglu theorem. The non-negative Borel measures $\mathbb{B}_+$ of total mass less than or equal to one form a weakly-* closed, hence weakly-* compact, and also convex subset of $\mathbb{B}$. Now the Krein-Milman theorem (see \cite{Rudin1991}[Theorem 3.23]) implies that $\mathbb{B}_+$ is equal to the closed convex hull of its extremal points, which are the Dirac measures in our case. Hence, for each $n\in\N$, there exists $J^n \in \N$ such that for every $j = 1,...,J^n$ there are
\[\theta_j^n \in (0,1], u_{0,j}^n \in \kappa_0^p\]
satisfying $\sum_{j=1}^{J^n} \theta_j^n = 1$ and
\[\mu_0^n := \sum_{j=1}^{J^n}\theta_j^n\delta_{u_{0,j}^n} \overset{\ast}{\rightharpoonup}\mu_0\,(n\to\infty)\]
on $C(\kappa_0^p)$. For each $n\in\N$ and $j \in \lbrace 1,...,n\rbrace$ let $u_j^n \in \CalU^p$ be a weak solution of the Euler equations with initial data $u_{0,j}^n$. The sequence of Borel probability measures on $C([0,T];H_{loc})$
\[\rho^n := \sum_{j=1}^{J^n}\theta_j^n\delta_{u_{j}^n}, n\in\N,\]  
then belongs already to $C(\kappa^p)'$, where $\kappa^p := (\Pi^{X^p}_0)^{-1}(\kappa_0^p) \cap \CalU^p$ is compact in $C([0,T];H_{loc})$ by \Cref{lem: comp Lp solutions}. Consequently, $C(\kappa^p)$ is separable and by the Banach-Alaoglu theorem there exists a Borel probability measure $\rho \in C(\kappa^p)'$ such that, after possibly passing to a subsequence, $\rho^n \overset{\ast}{\rightharpoonup} \rho\,(n\to\infty)$. In particular, $\rho$ is a $\CalU^p$-trajectory statistical solution as it is carried by the compact set $\kappa^p$ in $\CalU^p$.\\
Now we show that $\Pi_0^{X^p}\rho = \mu_0$. On the one hand, due to the very definition of the measures $\rho^n$, $n\in\N$,
\[\Pi^{X^p}_0\rho^n = \mu_0^n \overset{\ast}{\rightharpoonup}\mu_0\,(n\to\infty)\]
on $C(\kappa_0^p)'$. On the other hand, for every $\varphi \in C(\kappa_0^p)$, we have
\begin{equation}
\begin{split}
\int_{\kappa_0^p}\varphi(u_0)\,d(\Pi^{X^p}_0\rho^n)(u_0) &= \int_{\kappa^p} \underbrace{(\varphi\circ\Pi^{X^p}_0)}_{\in C(\kappa^p)}(u)\,d\rho^n(u)\\
&\to \int_{\kappa^p} (\varphi\circ\Pi^{X^p}_0)(u)\,d\rho(u)\,(n\to\infty)\\
&= \int_{\kappa_0^p} \varphi(u_0)\,d(\Pi^{X^p}_0\rho)(u_0),
\end{split}
\end{equation}
i.e.\ $\Pi^{X^p}_0\rho^n \overset{\ast}{\rightharpoonup} \Pi^{X^p}_0\rho\,(n\to\infty)$ in $C(\kappa_0^p)'$. Due to uniqueness of the weak-* limit, we obtain equality of $\Pi^{X^p}_0\rho$ and $\mu_0$ as measures on $\kappa_0^p$. But then equality also holds as measures on $X^p = H_{loc}$.\\
Now we consider the case where $\mu_0$ is an arbitrary Borel probability measure on $X_0^p$, not concentrated on an element of $\CalK'(X_0^p)$. By \Cref{lem: inner reg vort}, $\mu_0$ is inner regular with respect to the family $\CalK'(X_0^p)$ and we may find a sequence $(\kappa_0^{p,n})_{n\in\N} \subset \CalK'(X_0^p)$ such that 
\[\mu_0\left(\bigcup_{n=1}^\infty\kappa_0^{p,n} \right) = 1.\]
Furthermore, we may assume $\mu_0(\kappa_0^{p,n+1}) > \mu_0(\kappa_0^{p,n}) > 0$ and $\kappa_0^{p,n+1}\supset \kappa_0^{p,n}$ for all $n\in\N$. Also, let $D_n := \kappa_0^{p,n} \setminus \kappa_0^{p,n-1}$ for all $n\in\N$ ($\kappa_0^{p,0} := \emptyset$). Then
\[\mu_0\left(\bigcup_{n=1}^\infty D_n \right) = \mu_0\left(\bigcup_{n=1}^\infty\kappa_0^{p,n} \right) = 1\]
and for every Borel measurable set $A$ in $X_0^p$, we have
\[\mu_0(A) = \mu_0\left(A \cap \bigcup_{n=1}^\infty D_n \right) = \sum_{n=1}^\infty \mu_0(A \cap D_n).\]
Since $\mu_0(D_n) > 0$ for all $n\in\N$, we may define the Borel probability measures
\[\mu_0^n := \frac{\mu_0(\cdot \cap D_n)}{\mu_0(D_n)}\]
for all $n\in\N$. For each $n\in\N$, $\mu_0^n$ is concentrated on $\kappa_0^{p,n}$ and, by the first part of this proof, we may find a $\CalU^p$-trajectory statistical solution $\rho^n$ satisfying $\Pi^{X^p}_0\rho^n = \mu_0^n$.\\
Now we define the Borel probability measure
\[\rho := \sum_{n=1}^\infty \mu_0(D_n)\rho^n.\]
Since for each $n\in\N$, the measure $\rho^n$ is carried by a measurable subset of $\CalU^p$, the same applies for $\rho$, which means that $\rho$ is a $\CalU^p$-trajectory statistical solution. So it only remains to show that $\Pi^{X^p}_0\rho = \mu_0$. Let $\varphi \in C_b(H_{loc})$. Then
\begin{align}
&\int_{H_{loc}} \varphi(u_0)\,d\Pi^{X^p}_0\rho(u_0) = \int_{C([0,T];H_{loc})} \varphi(u(0))\,d\rho(u)\\
=&\sum_{n=1}^\infty \int_{C([0,T];H_{loc})}\mu_0(D_n) \varphi(u(0))\,d\rho^n(u) = \sum_{n=1}^\infty \mu_0(D_n)\int_{H_{loc}}\varphi(u_0)\,d\Pi^{X^p}_0 \rho^n(u_0)\\
=& \sum_{n=1}^\infty \mu_0(D_n)\int_{H_{loc}} \varphi(u_0)\,d\mu_0^n(u_0) = \int_{H_{loc}}\varphi(u_0)\,d\mu_0(u_0),
\end{align}  
which yields the claim.\\
We remark here that in all cases, \eqref{eq: energy inequ traj stats sols} can be proved same as \eqref{eq: energy inequ traj stats sols infty} or \eqref{eq: energy est traj stats sol infty}, based on the deterministic inequality \eqref{eq: energy ineq finite kin part}. Estimate \eqref{eq: energy inequ vort traj stats sol} also follows analogously to \eqref{eq: energy est traj stats sol infty} based on \eqref{eq: a priori est vort} and a measurability result for the $p$-norm of the vorticity, similar to that of \Cref{lem: measurability lp-norm vorticity}.\\
Next, we consider the case of $\mu_0$ being a Borel probability measure on $X_0^{VS}$. We apply \Cref{thm: ex traj stats sol v2}. Condition (\textit{H1'}) is satisfied by Delort's theorem (see \Cref{sec: euler}). Condition (\textit{H2'}) of inner regularity of $\mu_0$ with respect to the family $\CalK'(X_0^{VS})$ is satisfied due to \Cref{lem: inner reg vort}.
The final condition (\textit{H3'}) in \Cref{thm: ex traj stats sol v2} was just proved in \Cref{lem: comp VS solutions}. Likewise to the previous case, estimate \eqref{eq: energy inequ vort traj stats sol} can be proved analogously to \eqref{eq: energy est traj stats sol infty} based on \eqref{eq: a priori est vort L1,VS} and the Borel measurability of the mapping $\mathbb{E} \to \R, u \mapsto \int_{\R^2}\omega(u_0)$ with respect to $\mathbb{E}$ or the topology of $H^{-L}_{loc}(\R^2;\R^2)$ (see \Cref{lem: borel algebras on E}), which may be proved similarly to \Cref{lem: measurability lp-norm vorticity}.\\
Finally, we also apply \Cref{thm: ex traj stats sol v2} to the case of $\mu_0$ being a Borel probability measure on $X_0^1$. Condition (\textit{H1'}) is satisfied as described in \Cref{sec: euler}, i.e.\ for every $u_0 \in X_0^1$ there exists a weak solution of the Euler equations in $\CalU^1$. Condition (\textit{H2'}) of inner regularity of $\mu_0$ with respect to the family $\CalK'(X_0^1)$ is satisfied by \Cref{lem: inner reg vort}.
The final condition (\textit{H3'}) in \Cref{thm: ex traj stats sol v2} was just proved in \Cref{lem: comp L1 solutions}. Estimate \eqref{eq: energy inequ vort traj stats sol} may also be proved similarly to \eqref{eq: energy est traj stats sol infty} based on \eqref{eq: a priori est vort} and a measurability result for the $1$-norm of the vorticity, similar to that of \Cref{lem: measurability lp-norm vorticity} in combination with \Cref{lem: borel algebras on E}.
\end{proof}

In the final part of this subsection, we also provide arguments by which $\CalU^p$, $\CalU^{VS}$ or $\CalU^1$-trajectory statistical solutions can alternatively be constructed by consideration of the inviscid limit of $\CalU^\nu_{NS}$-trajectory statistical solutions as $(\nu \to 0)$. Our strategy here will be to first consider the $\CalU^\nu_{NS}$-trajectory statistical solutions as probability measures on a compact set, similar to those in $\CalK'(X_0^p), \CalK'(X_0^{VS})$ or $\CalK'(X_0^1)$, uniformly in $\nu$ and obtain a measure as a weak-* limit using the Banach-Alaoglu theorem. Then we argue that this measure does have support in the desired class of weak solutions so that it will ultimately be a trajectory statistical solution as desired.

\begin{lem}\label{lem: supp weak conv}
Let $M$ be a compact metric space and $(\pi_n)_{n\in\N}$ be a sequence of measures on the Borel-$\sigma$-algebra of $M$ converging weakly-* to some measure $\pi$. For every $x \in M$ denote by $\CalO_x$ its system of open neighbourhoods. Then
\begin{align}
\supp\pi &\subset \liminf_{n\to\infty} (\supp\pi_n)\\
&:= \lbrace x \in M: \forall O \in \CalO_x \text{ and all but finitely many }n\in\N: O \cap \supp\pi_n \neq \emptyset\rbrace\\
&= \lbrace x \in M: x \text{ is the limit of a sequence }(x_n)_{n\in\N}\in \prod_{n=1}^\infty\supp\pi_n \rbrace. 
\end{align}
\end{lem}

\begin{proof}
Let $x\in\supp\pi$ and consider an open ball $B_r(x)$ with radius $r > 0$, centred at $x$.\\ Let $\varphi \in C(M)$ be given by
\[\varphi(y) = \max\bigg\lbrace 0 , 1 - \frac{2}{r}\dist(y,B_{r/2}(x)) \bigg\rbrace, y \in M.\]
Then, as $B_{r/2}(x)$ is an open neighbourhood of $x$, we have
\[0 < \pi(B_{r/2}(x)) \leq \int_M \varphi\,d\pi = \lim_{n\to \infty} \int_M \varphi\,d\pi_n \leq \liminf_{n\to\infty}\pi_n(B_r(x)).\]
Therefore, for $n \in \N$ large enough, $\pi_n(B_r(x)) > 0$. This does suffice to show that $B_r(x) \cap \supp\pi_n \neq \emptyset$ for such an $n \in \N$. Indeed, suppose the opposite was true, i.e.\ $B_r(x) \cap \supp\pi_n = \emptyset$. Then for every $y \in B_r(x)$, there exists an open neighbourhood $O_y$ of $y$ in $B_r(x)$ such that $\pi_n(O_y) = 0$. For every $\varepsilon > 0$ small enough, the sets $(O_y)_{y \in \overline{B}_{r-\epsilon}(x)}$ form an open cover of the compact ball $\overline{B}_{r-\varepsilon}(x)$. Hence, for finitely many $y_1^\varepsilon,...,y^\varepsilon_{N_\varepsilon} \in B_r(x)$, $\overline{B}_{r-\varepsilon}(x) \subset \bigcup_{j=1}^{N_\varepsilon} O_{y_j^\varepsilon}$. In particular $\pi_n(\overline{B}_{r-\varepsilon}(x)) = 0$ by subadditivity. From $\sigma$-continuity, we arrive at the contradiction $\pi_n(B_r(x)) = 0$.\\
Hence, $B_r(x) \cap \supp\pi_n \neq \emptyset$ for all but finitely many $n\in\N$, which suffices as the balls $(B_r(x))_{r>0}$ form a basis of the neighbourhood system $\CalO_x$.
\end{proof}

To be able to tell that the support of the $\CalU_{NS}^\nu$-trajectory statistical solutions are indeed contained in $\CalU_{NS}^\nu$, we may show closedness of this space. The following lemma will be of use in the following sections. 

\begin{lem}\label{lem: comp NSE E}
Let $\kappa_0$ be a compact set in $\mathbb{E}$. Then $\kappa := (\Pi_0^\mathbb{E})^{-1}(\kappa_0) \cap \CalU_{NS}^\nu$ is compact in $C([0,T];\mathbb{E})$.
\end{lem}

\begin{proof}
Let $(u^n)_{n\in\N}$ be any sequence in $\kappa$. From combining the relative energy inequality \eqref{eq: stab est smooth sols} with the gradient control \eqref{eq: stab est grad smooth sols} (see also the comments thereafter), the compactness of the initial data in $\kappa_0$ yields the existence of a subsequence of $(u^n_{kin})_{n\in\N}$, which is Cauchy in $C([0,T];H)$. Then the corresponding subsequence of $(u^n)_{n\in\N}$ will be Cauchy in $C([0,T];\mathbb{E})$, whose limit $u \in C([0,T];\mathbb{E})$ has initial data in $\kappa_0$. From closedness of $\CalU_{NS}^\nu$ in $C([0,T];\mathbb{E})$ (\Cref{lem: meas U infty nu}) we obtain $u \in \kappa$ as desired.
\end{proof}

\begin{lem}\label{lem: supp traj stats sol NSE}
Let $\mu_0$ be a Borel probability measure on $\mathbb{E}$ which is concentrated on some compact set $\kappa_0$ in $\CalK'(X_0)$. Then the support of $S^{\nu}_{NS}\mu_0$ with respect to $C([0,T];\mathbb{E})$ is contained in $(\Pi^\mathbb{E}_0)^{-1}(\kappa_0) \cap \CalU^\nu_{NS}$. 
\end{lem}

\begin{proof}
We recall that in each case $\kappa_0$ is in fact compact in $\mathbb{E}$. \Cref{lem: comp NSE E} then implies that $(\Pi^\mathbb{E}_0)^{-1}(\kappa_0) \cap \CalU_{NS}^\nu$ is compact in $C([0,T];\mathbb{E})$ and in particular closed.\\
Since in each case this set has measure one with respect to $S_{NS}^\nu\mu_0$, we conclude the proof.
\end{proof}

\begin{thm}\label{thm: traj inviscid limit}
Let $\mu_0$ be a Borel probability measure on $X_0$. Then there exists a sequence $\nu^n \searrow 0$ such that $(S_{NS}^{\nu^n}\mu_0)_{n\in\N}$ converges to some Borel probability measure $\rho$ on $C([0,T];X)$ in the sense that 
\begin{equation}
\int_{C([0,T];X)} \varphi(u)\,dS_{NS}^{\nu^n}\mu_0(u) \to \int_{C([0,T];X)} \varphi(u)\,d\rho(u)\,(n \to \infty)
\end{equation}
for all bounded and continuous functions $\varphi$ on $C([0,T];X)$.\\
Moreover, $\rho$ is a $\CalU$-trajectory statistical solution of the Euler equations satisfying $\Pi_0^{X}\rho = \mu_0$.  
\end{thm}

\begin{proof}
We will discuss all three cases simultaneously. First, we suppose that $\mu_0$ is concentrated on some compact set $\kappa_0 \in \CalK'(X_0)$. The right-hand side in each of the a priori estimates demanded in the definition of $\CalU$ does not depend on the viscosity $\nu$, i.e.\ they are uniform. Therefore, if we let $\kappa \subset (\Pi^X_0)^{-1}(\kappa_0) \subset C([0,T];X)$ be the set of functions with initial data in $\kappa_0$, which satisfy the same a priori estimates demanded in the definition of $\CalU$, then as in \Cref{lem: comp Lp solutions}, \Cref{lem: comp VS solutions} or \Cref{lem: comp L1 solutions}, using the Aubin-Lions lemma or the Arzel\`{a}-Ascoli theorem yields the compactness of $\kappa$ in $C([0,T];X)$.\\
Due to the Banach-Alaoglu theorem, there exists a sequence $\nu^n \searrow 0$ and a Borel probability measure $\rho$ on $\kappa$ such that $\rho_{NS}^{\nu^n} := S_{NS}^{\nu^n}\mu_0 \overset{\ast}{\rightharpoonup} \rho$ in $C(\kappa)'$. The measure $\rho$ may be seen as a measure on $C([0,T];X)$ due to the Borel measurability of $\kappa$. In particular, the convergence of $(S_{NS}^{\nu^n}\mu_0)_{n\in\N}$ towards $\rho$ holds in the sense described in the theorem.\\
We now show that $\rho$ is indeed a trajectory statistical solution satisfying $\Pi_0^X\rho = \mu_0$.\\
By \Cref{lem: supp traj stats sol NSE}, the support of $\rho_{NS}^{\nu^n}$ with respect to $C([0,T];\mathbb{E})$ is contained in $(\Pi_0^\mathbb{E})^{-1}(\kappa_0) \cap \CalU_{NS}^{\nu^n}$ for every $n \in \N$, which is compact in $C([0,T];\mathbb{E})$ due to \Cref{lem: comp NSE E}. In general, the support is of course depending on the considered topology and not just the Borel-$\sigma$-algebra. However, as we have the compact embeddings $\mathbb{E} \hookrightarrow H_{loc} \hookrightarrow H^{-L}_{loc}(\R^2;\R^2)$, for any $n \in \N$, $(\Pi_0^\mathbb{E})^{-1}(\kappa_0) \cap \CalU_{NS}^{\nu^n}$ is compact in $C([0,T];X)$. Consequently, the support in $C([0,T];X)$ of $\rho_{NS}^{\nu^n}$ will be contained in $(\Pi_0^\mathbb{E})^{-1}(\kappa_0) \cap \CalU_{NS}^{\nu^n}$ in all three cases and any $n\in\N$.\\
Let $u \in \supp\rho \subset \kappa$. By the point we just made about the support of $(\rho_{NS}^{\nu^n})_{n\in\N}$ and \Cref{lem: supp weak conv}, there exists a sequence $(u_{NS}^{\nu^n})_{n\in\N}$ such that $u_{NS}^{\nu^n} \in (\Pi_0^\mathbb{E})^{-1}(\kappa_0) \cap \CalU^{\nu^n}_{NS}$ for all $n\in\N$ and $u_{NS}^{\nu^n} \to u\,(n\to\infty)$ in $C([0,T];X)$.\\
To conclude that $u$ is a weak solution of the Euler equations in the desired class $\CalU^p$, $\CalU^{VS}$ or $\CalU^1$, one would need to consider each case individually as we have done for \Cref{lem: comp Lp solutions}, \Cref{lem: comp VS solutions} and \Cref{lem: comp L1 solutions}. The arguments provided in those lemmata can be used analogously here mainly for three reasons: First, for the convergence to $0$ of the term involving the viscosity, one only needs a uniform bound on the local kinetic energy for the Navier-Stokes equations, which holds due to \eqref{eq: comp bdd L2_loc weak sols} and by $\kappa_0$ being compact in $\mathbb{E}$. Second, the a priori estimates to extract a subsequence which converges in an appropriate way are likewise satisfied in each case by the Navier-Stokes equations as outlined in \Cref{sec: euler}. Finally, $u_{NS}^{\nu^n}(0) \in \kappa_0$ for every $n \in \N$, which allows us to conclude that the a priori estimates demanded in the definition of $\CalU^p$, $\CalU^{VS}$ and $\CalU^1$ are satisfied by $u$ in the respective case. This latter fact of strong convergence of the initial vorticities also implies in the case $X = X^p$, due to \Cref{lem: stab renormalized sols}, that $\omega(u)$ will be a renormalized solution of the vorticity formulation of the Euler equations (see \Cref{lem: stab renormalized sols}).\\
Omitting further details, we conclude that $\rho$ is a trajectory statistical solution on the considered set of weak solutions of the Euler equations.\\
It remains to show $\Pi_0^X\rho = \mu_0$. Note that for any continuous function $\varphi$ on $X$, $\varphi \circ \Pi^X_0$ is continuous on $C([0,T];X)$, from which we immediately derive
\begin{align}
&\int_{X} \varphi(u_0)\,d\Pi^X_0\rho(u_0) = \int_{C([0,T];X)}(\varphi \circ \Pi^X_0)(u)\,d\rho(u) = \int_{\kappa}(\varphi \circ \Pi^X_0)(u)\,d\rho(u)\\
=& \lim_{n \to \infty} \int_{\kappa}(\varphi \circ \Pi^X_0)(u)\,dS_{NS}^{\nu^n}\mu_0(u) = \int_{\kappa_0} \varphi(u_0)\,d\mu_0(u_0) = \int_{X} \varphi(u_0)\,d\mu_0(u_0).
\end{align}
Hence, $\Pi^X_0\rho = \mu_0$ as desired.\\
The general case when $\mu_0$ is no longer concentrated on some set $\kappa_0 \in \CalK'(X_0)$ is quite similar to that in \Cref{thm: ex traj stats sol Euler}. By \Cref{lem: inner reg vort}, $\mu_0$ is inner regular with respect to the family $\CalK'(X_0)$ and we may find a sequence $(\kappa_0^n)_{n\in\N} \subset \CalK'(X_0)$ such that 
\[\mu_0\left(\bigcup_{n=1}^\infty\kappa_0^n \right) = 1.\]
We may assume $\mu_0(\kappa_0^{n+1}) > \mu_0(\kappa_0^{n}) > 0$ and $\kappa_0^{n+1}\supset \kappa_0^{n}$ for all $n\in\N$. Also, let $D_n := \kappa_0^{n} \setminus \kappa_0^{n-1}$ for all $n\in\N$ ($\kappa_0^{0} := \emptyset$). Then
\[\mu_0\left(\bigcup_{n=1}^\infty D_n \right) = \mu_0\left(\bigcup_{n=1}^\infty\kappa_0^{n} \right) = 1\]
and for every Borel measurable set $A$ in $X_0$, we have
\[\mu_0(A) = \mu_0\left(A \cap \bigcup_{n=1}^\infty D_n \right) = \sum_{n=1}^\infty \mu_0(A \cap D_n).\]
Since $\mu_0(D_n) > 0$ for all $n\in\N$, we may define the Borel probability measures
\[\mu_0^n := \frac{\mu_0(\cdot \cap D_n)}{\mu_0(D_n)}\]
for all $n\in\N$. For each $n\in\N$, $\mu_0^n$ is concentrated on $\kappa_0^{n}$ and by the first part of this proof and a diagonal sequence argument, we can find a sequence $\nu^n \searrow 0$ such that $(S_{NS}^{\nu^n}\mu_0^k)_{n\in\N}$ converges to $\rho^k$ for every $k \in \N$ in the sense described in the statement of the theorem,
where $\rho^k$ is a trajectory statistical solution on the considered class, satisfying $\Pi_0^X\rho^k = \mu_0^k$.\\
Next, we now show that $(S_{NS}^{\nu^n}\mu_0)_{n\in\N}$ converges to $\rho := \sum_{k=1}^\infty \mu_0(D_k)\rho^k$ in the described way. In before, we mention that $\rho$ is a trajectory statistical solution of the Euler equations satisfying $\Pi_0^X\rho = \mu_0$, which can be shown identically to \Cref{thm: ex traj stats sol Euler}.\\
Let $\varphi$ be a bounded continuous function on $C([0,T];X)$ and let $\varepsilon > 0$. Then there exists $k_0 \in \N$ such that for every $k \in \N$ satisfying $k \geq k_0$
\[\mu_0\left( \bigcup_{l=1}^k D_l \right) \geq 1- \varepsilon.\]
For $k \geq k_0$, we thereby obtain
\begin{align}
&\limsup_{n \to \infty} \bigg| \int_{C([0,T];X)}\varphi(u)\,dS_{NS}^{\nu^n}\mu_0(u) - \int_{C([0,T];X)}\varphi(u)\,d\rho(u)\bigg|\\
\leq& \limsup_{n\to\infty}\sum_{l=1}^k\mu_0(D_l) \bigg|\int_{C([0,T];X)} \varphi(u)\,dS_{NS}^{\nu^n}\mu_0^l(u) - \int_{C([0,T];X)} \varphi(u)\,d\rho^l(u)\bigg|\\
&\, + 2\|\varphi\|_\infty\sum_{l=k+1}^\infty\mu_0(D_l)\\
=& 2\|\varphi\|_{\infty}\varepsilon.
\end{align}
\end{proof}

\section{Statistical solutions of the Euler equations in phase space}

In this section we construct phase space statistical solutions of the Euler equations by projecting the trajectory statistical solutions, which we constructed in the previous section, in time as suggested by
\Cref{thm: phase space stats sol}.\\
So far, we have worked with the velocity formulation of the (deterministic) Euler and Navier-Stokes equations and will continue to do so here. As the vorticity does not explicitly play a role in this formulation, we may actually use more or less the same definition of phase space statistical solution of the Euler equations in all considered cases, by which we also mean the usage of the same class of test functionals throughout, based on the set of divergence-free test functions $C_c^\infty(\R^2;\R^2) \cap H$. We remark though that for instance when considering the Yudovich class, one could relatively effortlessly use a larger class of test functionals due to more integrability of the velocity field. As it makes little to no difference for our work here, we will restrain ourselves from doing so.\\

Throughout this section, we let $Y = \CalD \cap H$ be the space of divergence-free test functions, considered as a subspace of $\CalD := C^\infty_c(\R^2;\R^2)$, endowed with the inductive limit topology (see \cite{Rudin1991}[Chapter 6]): We define the spaces $\CalD_K := \lbrace \varphi \in C^\infty(\R^2;\R^2) : \supp\varphi \subset K\rbrace$ for every compact set $K \subset \R^2$ as Fr\'{e}chet spaces with the topology $\tau_K$ coming from the norms $\|\cdot\|_N := \max_{|\alpha| \leq N} \|\partial^\alpha\cdot\|_{L^\infty(\R^2)}$, $N\in\N$. On $\CalD$, we then consider the topology associated to the local basis of the origin of convex, balanced sets $W \subset \CalD$ such that $\CalD_K \cap W \in \tau_K$ for every compact set $K \subset \R^2$.\\
On $Y'$, we consider the strong topology of topological dual spaces, i.e.\ the topology of uniform convergence on bounded sets. If $E$ is a bounded subset of $Y$, then $E \subset \CalD_K$ for some $K \subset \R^2$ and there are numbers $M_N < \infty$ such that
\begin{equation}\label{eq: bd on Y}
\|\varphi\|_N \leq M_N
\end{equation}
for all $N \in \N, \varphi \in E$.\\
Throughout all considered cases, in the setting of \Cref{def: phase space stats sol}, we will choose $Z = \mathbb{E}$ and $X$ will be one of the spaces $X^\infty = \mathbb{E}, X^p = H_{loc}$ or $X^{VS} = X^1 = H^{-L}_{loc}(\R^2;\R^2)$ as in the previous section. We are likewise going to use $X_0$ and $\CalU$ as general placeholders for the times we wish to prove or formulate a statement that is valid in all cases.\\ 
We have the continuous embeddings
\[\mathbb{E} \hookrightarrow X \hookrightarrow Y'_{w*}\]
and we are going to consider the function $F: \mathbb{E} \to Y'$, given by
\[\langle F(u),v \rangle_{Y',Y} = \int_{\R^2} (u \otimes u) : \nabla v\,dx\]
for every $u \in \mathbb{E}, v \in Y$.\\
Clearly, $u \in L^\infty(0,T;\mathbb{E}) \cap C([0,T];X)$ is a weak solution of the Euler equations in the sense of \Cref{def: Euler weak vel} if and only if $u_t = F(u)$ in the weak sense described in \Cref{thm: phase space stats sol}. We note that compared to the introductory section on the abstract framework of statistical solutions, we omitted the time interval $[0,T]$ from the domain of $F$, as $F$ does not explicitly depend on the time.

\begin{lem}\label{lem: meas F}
The mapping $F: \mathbb{E} \to Y'$ is $\CalB(\mathbb{E})$-$\CalB(Y')$ measurable. Moreover, the corresponding Nemytskii operator $(t,u) \mapsto F(u(t))$ is $\CalL([0,T]) \otimes \CalB(C([0,T];X))$-$\CalB(Y')$ measurable when $F$ is extended by $0$ from $\mathbb{E}$ to $X$.
\end{lem}

\begin{proof}
We show that $F$ is even continuous. As $\mathbb{E}$ is a normed space, it suffices to show sequential continuity. Let $(u^n)_{n\in\N}$ be a converging sequence in $\mathbb{E}$ with limit $u$. We now show $\lim_{n\to\infty} F(u^n) = F(u)$ in $Y'$. Let $B \subset Y$ be bounded. By \eqref{eq: bd on Y}, there exists some compact set $K \subset \R^2$ and $M_1 > 0$ such that for all $\varphi \in B$, we have
\[\supp\varphi \subset K \text{ and } \|\nabla\varphi\|_{L^\infty(\R^2)} \leq M_1.\]
Consequently, 
\[\sup_{\varphi \in B} |\langle F(u^n) - F(u),\varphi \rangle_{Y',Y}| \leq \|(u^n \otimes u^n) - (u \otimes u)\|_{L^1(K)} M_1.\]
The right-hand side converges to $0$ since convergence in $\mathbb{E}$ particularly implies convergence in $L^2_{loc}(\R^2;\R^2)$ so that $(u^n \otimes u^n)_{n\in\N}$ converges in $L^1_{loc}(\R^2;\R^{2\times 2})$.\\
The measurability of the corresponding Nemytskii operator is a general result (see Proposition 2.1 in \cite{Bronzi2016}).
\end{proof}

As before, $\Pi_t^X: C([0,T];X) \to X$ denotes the time evaluation mapping at time $t \in [0,T]$.

We are now going to construct projected phase space statistical solutions. In the Yudovich class, we may construct them in an elementary way using the solution operator $S^\infty: X_0^\infty \to C([0,T];X^\infty)$ that we introduced in  
\Cref{sec: traj sols infty and NS}. The arguments are very close to the ones used in the proof of \Cref{thm: ex phase space stats sol} and \Cref{thm: ex phase space stats sol v2} in \cite{Bronzi2016}, to which we will refer in the other cases.\\
In the following, for $u \in X$, $\|\omega(u)\|_X$ is defined as in \eqref{eq: norm vorticity} if $X=X^p,X^{VS},X^1$ and $\|\omega(u)\|_{X^\infty} := \|\omega(u)\|_{(L^1\cap L^\infty)(\R^2)}$. 

\begin{thm}\label{thm: ex phase space stats sol infty}
Let $\mu_0$ be a Borel probability measure on $X_0$ satisfying
\begin{equation}\label{eq: integrability infty mu_0}
\int_{X_0} \gamma(u_0)^2\,d\mu_0(u_0) < \infty
\end{equation}
with $\gamma$ as in \eqref{eq: comp bdd L2_loc weak sols}.\\ Let $\rho$ be a $\CalU$-trajectory statistical solution satisfying $\Pi_0^X\rho = \mu_0$. Then the family of Borel probability measures $\lbrace \rho_t \rbrace_{0 \leq t \leq T}, \rho_t = \Pi_t^X\rho$ for every $0 \leq t \leq T$, is a statistical solution in phase space of the equation $u_t = F(u)$, satisfying $\rho_0 = \mu_0$ as well as
\begin{equation}\label{eq: energy inequ phase space stats sols}
\int_{\mathbb{E}}\|u_{kin}\|_{L^2(\R^2)}\,d\rho_t(u) \leq \int_{\mathbb{E}}a^{|m(u_0)|}\|u_{0,kin}\|_{L^2(\R^2)}\,d\mu_0(u_0)
\end{equation}
for almost every $0 \leq t \leq T$ in the cases $X = X^p, X^{VS}, X^1$ and every $0 \leq t \leq T$ if $X = X^\infty$.\\
Moreover, if $\mu_0$ satisfies $\int_{X} \|\omega(u_0)\|_{X}\,d\mu_0(u_0) < \infty$, then
\begin{equation}\label{eq: energy est phase space stats sol}
\int_{X} \|\omega(u)\|_{X}\,d\rho_t(u) \leq \int_{X} \|\omega(u_0)\|_{X}\,d\mu_0(u_0)
\end{equation}
for every $0 \leq t \leq T$, with equality if $X = X^\infty$ or $X= X^p$.
\end{thm}

\begin{proof}
We begin with the case $X = X^\infty = \mathbb{E}$. Let $\rho^\infty = S^\infty\mu_0$ be the $\CalU^\infty$-trajectory statistical solution from \Cref{thm: ex traj stats sol infty} so that $\rho^\infty_t = \Pi^\mathbb{E}_t\rho^\infty$ for all $0 \leq t \leq T$.
We will check in an elementary way in this case that $\lbrace \rho_t^\infty\rbrace_{0\leq t \leq T} = \lbrace (\Pi_t^\mathbb{E} \circ S^\infty)\mu_0\rbrace_{0\leq t \leq T}$ satisfies all four conditions in \Cref{def: phase space stats sol}. 
\begin{enumerate}
\item[i)] Let $\varphi \in C_b(\mathbb{E})$. We obtain from the Lebesgue dominated convergence theorem that
\[t \mapsto \int_{\mathbb{E}}\varphi(u)\,d\rho_t(u) = \int_{C([0,T];\mathbb{E})}\varphi(u(t))\,d\rho(u)\]
is continuous and obviously bounded.
\item[ii), iii)] Since solutions in $\CalU^\infty$ are in $C([0,T];\mathbb{E})$, it is clear that $\rho_t$ is carried by $Z = \mathbb{E}$ for every $0 \leq t \leq T$. As we have already shown sufficient measurability properties of $F$ and its associated Nemytskii operator in \Cref{lem: meas F}, to verify ii) and iii) in \Cref{def: phase space stats sol}, we only need to show that $F$ is bounded appropriately by integrable functions, which we will do in one step.\\
Let $0 \leq t \leq T$ and $u_0 \in X_0^\infty$ with associated weak solution of the Euler equations $u = S^\infty(u_0) \in \CalU^\infty$. For any $v \in Y$, there exists $r > 0$ such that $v$ has compact support in $B_r(0)$. Then we may estimate using \eqref{eq: comp bdd L2_loc weak sols} 
\begin{align}
|\langle F^\infty(u(t)),v\rangle_{Y,Y'}|&\leq \|\nabla v\|_{L^\infty(\R^2)} \|u(t)\|^2_{L^2(B_r(0))}\\
&\leq C\|\nabla v\|_{L^\infty(\R^2)}\max\lbrace1,r^2\rbrace\gamma(u_0)^2
\end{align}
with the right-hand side being independent of $t$. Therefore, the integrability assumption \eqref{eq: integrability infty mu_0} shows that $|\langle F(u(t)),v\rangle_{Y',Y}|$ can even be dominated by a $dt \otimes \rho$ integrable function on $[0,T]\times C([0,T];\mathbb{E})$. As $\rho_t$ is given as the projection of $\rho$ at time $t$, this shows that both ii) and iii) hold.
\item[iv)] Finally, let $\Phi$ be a cylindrical test function in $Y'$, i.e.\ there exists $\phi \in C^1_c(\R^k)$ for some $k \in \N$ and $v_1,...,v_k \in Y$, such that
\[ \Phi(w) = \phi(\langle w,v_1\rangle_{Y',Y},...,\langle w,v_k\rangle_{Y',Y}) \]
for all $w \in Y'$. Let $u \in \CalU^\infty$ be a weak solution. We note that for almost every $0 \leq s \leq T$, we have
\begin{equation}
\begin{split}
&\frac{d}{ds}\Phi(u(s))\\
=& \frac{d}{ds}\phi(\langle u(s),v_1\rangle_{Y',Y},...,\langle u(s),v_k\rangle_{Y',Y})\\
=& \sum_{j=1}^k \partial_{x_j} \phi(\langle u(s),v_1\rangle_{Y',Y},...,\langle u(s),v_k\rangle_{Y',Y})\frac{d}{ds}\langle u(s),v_j \rangle_{Y',Y}\\
=& \sum_{j=1}^k \partial_{x_j} \phi(\langle u(s),v_1\rangle_{Y',Y},...,\langle u(s),v_k\rangle_{Y',Y})\int_{\R^2} (u(s)\otimes u(s)) : \nabla v_j\,dx\\
=& \langle u(s)\otimes u(s),\sum_{j=1}^k \partial_{x_j} \phi(\langle u(s),v_1\rangle_{Y',Y},...,\langle u(s),v_k\rangle_{Y',Y})\nabla v_j \rangle_{L^2(\R^2),L^2(\R^2)}\\
=& \langle u(s)\otimes u(s),\nabla\Phi'(u(s))\rangle_{L^2(\R^2),L^2(\R^2)}\\
=& \langle F(u(s)),\Phi'(u(s))\rangle_{Y',Y}.
\end{split}
\end{equation}
Integrating from $t'$ to $t$, $0 \leq t' \leq t \leq T$, then yields
\[\Phi(u(t)) = \Phi(u(t')) + \int_{t'}^t \langle F(u(s)),\Phi'(u(s))\rangle_{Y',Y}\,ds.\]
Due to the measurability of $F$ shown in \Cref{lem: meas F} and the estimates in part ii), we may integrate with respect to $\mu_0$, apply Fubini's theorem and use the definition of $\lbrace \rho_t \rbrace_{0\leq t \leq T}$ to obtain
\[\int_{\mathbb{E}} \Phi(u)\,d\rho_t(u) = \int_{\mathbb{E}}\Phi(u)\,d\rho_{t'}(u) + \int_{t'}^t\int_{\mathbb{E}} \langle F(u),\Phi'(u)\rangle_{Y',Y}\,d\rho_s(u)\,ds.\]
\end{enumerate}
Estimates \eqref{eq: energy inequ phase space stats sols} and \eqref{eq: energy est phase space stats sol} are immediate consequences of \eqref{eq: energy inequ traj stats sols infty} and \eqref{eq: energy est traj stats sol infty}.\\
For the remaining cases, we will directly apply \Cref{thm: ex phase space stats sol v2}. Conditions (\textit{H1'}) - (\textit{H3'}) have been checked in preparation for \Cref{thm: ex traj stats sol Euler}. Condition (\textit{H4}) is satisfied by \Cref{lem: borel algebras on E}. (\textit{H5}) follows from $\CalU \subset L^\infty(0,T;\mathbb{E})$, our choice of $F$, \Cref{lem: meas F} and the weak formulation of the Euler equations. Finally, (\textit{H6}) can be verified as we just did in step ii) in the case of $X = X^\infty = \mathbb{E}$.\\
Finally, \eqref{eq: energy est phase space stats sol} and \eqref{eq: energy inequ phase space stats sols} are consequences of \eqref{eq: energy inequ traj stats sols} and \eqref{eq: energy inequ vort traj stats sol}.
\end{proof}

\begin{remark}
From the uniqueness of the $\CalU^\infty$-trajectory statistical solutions we can immediately conclude that $\lbrace \rho_t\rbrace_{0\leq t \leq T}$ in \Cref{thm: ex phase space stats sol infty} is the unique projected statistical solution in phase space in that specific case.\\
However, there may be other statistical solutions in phase space which cannot be obtained from projecting a $\CalU^\infty$-trajectory solution. Formally, computations as in \cite{Foias2001} in the case of the two-dimensional Navier-Stokes equations on bounded domains, or adaptations for the case of $\R^2$, as indicated in \cite{Kelliher2009}, may be used to show uniqueness among all statistical solutions in phase space. However, making those computations rigorous requires some kind of Fr\'{e}chet-differentiable dependence on the initial data. In \cite{Foias2001}, this problem was overcome with quite some effort by considering the Galerkin approximations which stem from an ordinary differentiable equation, where differentiable dependence on the initial data clearly holds. 
\end{remark}

To close this section, let us briefly point out that the inviscid limit results \Cref{thm: traj inviscid limit infty} and \Cref{thm: traj inviscid limit} also yield analogous results for the phase space statistical solutions.\\
For phase space statistical solutions of the Navier-Stokes equations with viscosity $\nu > 0$ in the framework of \cite{Bronzi2016}, we let $Z =  X = \mathbb{E}$ and choose $Y$ as in the previously considered case of phase space statistical solutions of the Euler equations. To be more in line with previous work on phase space statistical solutions of the Navier-Stokes equations, in particular \cite{Kelliher2009}, we could also choose $X = \mathbb{E}$ and $Z_{NS} = \lbrace u \in \mathbb{E} : \nabla u \in L^2(\R^{2 \times 2})\rbrace,$ endowed with the norm $\|\cdot\|_{Z_{NS}} := \|\cdot\|_{\mathbb{E}} + \|\nabla\cdot\|_{L^2(\R^{2 \times 2})}$ and for the space of test functions let $Y_{NS} = H^1_c(\R^2;\R^2) \cap H$ with a similar type of inductive limit topology. We remark that the different choice here makes little difference since in either case, the projected phase space statistical solution is given by $\lbrace(S_{NS}^\nu\circ\Pi_t^\mathbb{E})\mu_0\rbrace_{0 \leq t \leq T}$ and the only difference is that we would view them as measures on different subspaces.\\
Let $F_{NS}^\nu: \mathbb{E} \to Y'$ be given by
\begin{equation}
\langle F_{NS}^\nu(u),v\rangle_{Y',Y} = \int_{\R^2} (u \otimes u) : \nabla v\,dx + \nu \int_{\R^2}  u \cdot \Delta v\,dx
\end{equation}
for every $u \in \mathbb{E}$ and $v \in Y$. A function $u \in C([0,T];\mathbb{E})$ with gradient $\nabla u \in L^2(0,T;L^2(\R^{2\times 2}))$ is a weak solution of the Navier-Stokes equations with viscosity $\nu$ in the sense of \Cref{def: NSE weak vel} if and only if $\partial_t u = F_{NS}^\nu(u)$ in the weak sense described in \Cref{thm: phase space stats sol}. Moreover, similarly to \Cref{lem: meas F}, $F_{NS}^\nu$ can be shown to be continuous so that all desired measurability properties are satisfied.

\begin{lem}\label{lem: meas F NSE}
The mapping $F_{NS}^\nu: \mathbb{E} \to Y'$ is $\CalB(\mathbb{E})$-$\CalB(Y')$ measurable. Moreover, the corresponding Nemytskii operator $(t,u) \mapsto F_{NS}^\nu(u(t))$ is $\CalL([0,T]) \otimes \CalB(C([0,T];\mathbb{E}))$-$\CalB(Y')$ measurable.
\end{lem}

We then obtain (see also Theorem 7.1 in \cite{Kelliher2009}).

\begin{thm}\label{thm: ex phase space stats sol NSE}
Let $\mu_0$ be a Borel probability measure on $\mathbb{E}$ satisfying
\begin{equation}\label{eq: integrability NSE mu_0}
\int_{\mathbb{E}} \gamma(u_0)^2\,d\mu_0(u_0) < \infty
\end{equation}
with $\gamma$ as in \eqref{eq: comp bdd L2_loc weak sols}. Then the family of Borel probability measures $\lbrace \rho^\nu_t \rbrace_{0 \leq t \leq T}, \rho^\nu_t = (\Pi^\mathbb{E}_t \circ S_{NS}^\nu)\mu$ for every $0 \leq t \leq T$, is a statistical solution in phase space of the equation $u_t = F_{NS}^\nu(u)$, satisfying $\rho^\nu_0 = \mu_0$.
\end{thm}

\begin{remark}
\begin{enumerate}[i)]
\item As we pointed out for the trajectory statistical solutions of the Navier-Stokes equations, one can also obtain an energy inequality for phase space statistical solutions of the Navier-Stokes equations as well as other properties, derived from the deterministic equations. As these serve no further purpose here, we again refer the reader to Sections $6$ and $7$ in \cite{Kelliher2009}.
\item In \cite{Kelliher2009}[Theorem 7.1], it has been shown that under the assumption of $\mu_0$ having support in $\lbrace u \in \mathbb{E} : \nabla u \in L^2(\R^{2\times 2})\rbrace$, bounded with respect to $\|\cdot\|_{\mathbb{E}}$, the phase space statistical solution in \Cref{thm: ex phase space stats sol NSE} is unique. The case of uniqueness of phase space statistical solutions with unbounded initial supports appears to be open. The assumption of a bounded support of the initial distribution could also not be completely omitted in the classical case, where the underlying domain is bounded, but only be relaxed by instead making further assumptions on $\lbrace \rho_t^\nu \rbrace_{0\leq t \leq T}$ (see for instance the discussion on p. 266 in \cite{Foias2001}). 
\end{enumerate}
\end{remark} 

Going back to the phase space statistical solutions of the Euler equations with phase space $X$, for any bounded continuous function $\varphi: X \to \R$, the composition $\varphi \circ \Pi^X_t$ is a real-valued, bounded continuous function on $C([0,T];X)$.\\
Therefore, we immediately obtain the following result from \Cref{thm: traj inviscid limit infty}, \Cref{thm: traj inviscid limit} and the Lebesgue dominated convergence theorem.

\begin{thm}\label{thm: phase space inviscid limit infty}
Let $\mu_0$ be a Borel probablity measure on $X_0$ satisfying
\begin{equation}
\int_{X_0} \gamma(u_0)^2\,d\mu_0(u_0) < \infty
\end{equation} 
with $\gamma$ as in \eqref{eq: comp bdd L2_loc weak sols}. Then there exists a $\CalU$-trajectory statistical solution $\rho$ satisfying $\Pi_0\rho = \mu_0$ and a sequence $\nu^n \searrow 0$ such that the projected phase space statistical solutions $\lbrace \Pi_t^X S^{\nu^n}_{NS}\mu_0 \rbrace_{0 \leq t \leq T}$ of the Navier-Stokes equations converge to a family of projected phase space statistical solutions $\lbrace \Pi^X_t\rho\rbrace_{0 \leq t \leq T}$ of the Euler equations, in the sense that for every bounded continuous function $\varphi$ on $X$ and every $0 \leq t \leq T$
\[\int_{X} \varphi(u)\,d(\Pi^X_t S_{NS}^{\nu^n}\mu_0)(u) \to \int_{X} \varphi(u)\,d(\Pi^X_t\rho)(u)\,(n \to \infty).\]
In the case $X = X^\infty$, this holds for every subsequence $\nu^n \searrow 0$.
\end{thm}

\begin{remark}
The inviscid limit result of phase space statistical solutions here is fundamentally different from previous work that we pointed at in the introduction of this article such as \cite{Chae1991p}, \cite{Constantin1997} or \cite{Kelliher2009}. Unlike we did here, in these articles, the phase space 
statistical solutions of the Euler equations have been constructed by an inviscid limit argument. Not only does this require the construction of a family of measures as some sort of limit of the time-parametrized measures that are the phase space statistical solutions of the Navier-Stokes equations. The sense in which this convergence holds needs to be strong enough to prove that each term in the Foias-Liouville equation \eqref{eq: Foias-Liouville equ} corresponding to the Navier-Stokes equations converges appropriately to conclude that this limit satisfies the Foias-Liouville equation corresponding to the Euler equations.\\
In our work, in the cases $X = \mathbb{E}$ or $X = H_{loc}$, we have convergence of the phase space statistical solutions in an $H_{loc}$ related sense, which would suffice to argue that each integral in the Foias-Liouville equation of the Navier-Stokes equations converges appropriately. However, this is not even necessary, as we do already know from the discussions on the inviscid limit for the trajectory statistical solutions that we obtain convergence of a subsequence to a projected phase space statistical solution of the Euler equations, where we know independently from this convergence that it satisfies the desired Foias-Liouville equation.
\end{remark}

\bibliography{Statistical_solutions_of_the_incompressible_Euler_equations.bib}
\bibliographystyle{abbrv}

\end{document}